\documentclass {siamltex}
\usepackage{makeidx}     
\usepackage{graphicx}
\usepackage{latexsym}
\usepackage{array}
\usepackage{amssymb}
\usepackage{amsmath}
\usepackage{amsbsy}

\newcommand{\tnorm}[1]{\vert\hspace{-0.3mm}\Vert#1\Vert\hspace{-0.3mm}\vert}

\newtheorem{remark}{\it Remark\/}
\usepackage{eepic}
\title{A penalty free non-symmetric Nitsche type method for the weak
  imposition of boundary conditions}

\author{Erik Burman\thanks{Department of Mathematics, 
University of Sussex, Brighton, 
UK--BN1 9QH, 
United Kingdom; ({\tt E.N.Burman@sussex.ac.uk})}
}
 
\begin{document}

\maketitle

\begin{abstract}
In this note we show that the non-symmetric version of the classical
Nitsche's method for the weak imposition of boundary conditions is
stable without penalty term. We prove optimal $H^1$-error estimates
and $L^2$-error estimates that are suboptimal with half an order in
$h$. Both the pure diffusion and the convection--diffusion problems
are discussed.
\end{abstract}


\newcommand{\cut}{c}
\def\IR{\mathbb R}
\def\Ext{\mbox{\textsf{E}}}

\section{Introduction}
In his seminal paper from 1971, \cite{Nit71}, Nitsche proposed a consistent penalty
method for the weak imposition of boundary conditions. The formulation
proposed was symmetric so as to reflect the symmetry of the underlying
Poisson problem. Stability was obtained thanks to a penalty term, with a
penalty parameter that must satisfy a lower bound to ensure coercivity. 

A non-symmetric version of Nitsche's method was later
proposed by Freund and Stenberg \cite{FS95} and it was noted that this method did not
need the lower bound for stability. The penalty term however could not 
be omitted, since coercivity fails, and error estimates degenerate as the penalty parameter
goes to zero. The non-symmetric version of Nitsche's method was then
proposed as a discontinuous Galerkin (DG) method by Oden et al., \cite{OBB98} and it was proven that 
the non-symmetric version was stable for polynomial orders $k \ge 2$,
by Girault et al. \cite{RWG01} and Larson and Niklasson \cite{LN04}. 
In \cite{LN04} stability for the penalty free case is proved using an inf-sup argument that relies
on the important number of degree's of freedom available in high order
DG-methods.

To the best of our knowledge no similar results have been proven for
the non-symmetric version of Nitsche's method for the imposition of
boundary conditions when continuous approximation spaces are used. 
Indeed in this case the DG-analysis does not work since polynomials 
may not be chosen independently on different elements because of the continuity
constraints. Weak impositition of boundary conditions has been
advocated by Hughes et al. for turbulence computations of LES-type in
\cite{BaHu07}. They showed that the mean flow in the boundary
layer was more accurately captured using weakly rather than strongly
imposed boundary conditions. They also noted that the
non-symmetric Nitsche's method appears stable without penalty
\cite{HEML00b}.

In applications there is interest in reducing the
number of free parameters used without increasing the number of degrees of freedom needed for
the coupling, see \cite{GW10} for a discussion. From this point of
view a penalty free Nitsche method is a welcome addition to the
computational
toolbox, in particular for flow problems where the system matrix is
non-symmetric anyway, because of the convection terms. It has no penalty parameter and does not make use of Lagrange
multipliers.

Numerical evidence also suggests that the unpenalized non-symmetric
Nitsche type method has
some further interesting properties. When using iterative solution
methods in domain decomposition it
has been shown to have more favorable convergence properties compared
to the symmetric method
\cite{BZ06}. For the solution of Cauchy-type inverse problem using
steepest descent type algorithms it has been shown numerically to have superior
convergence properties in the initial phase of the iterations compared to the symmetric version or strongly
imposed conditions, in spite of the lack of dual consistency. 

In view of this
the question naturally arises if the penalty free method is sound, or if it can
fail under unfortunate circumstances.

In this paper we prove  for the Poisson problem that the non-symmetric 
Nitsche's method is indeed stable and optimally convergent in the
$H^1$-norm for polynomial orders $k \ge 1$. We also show that in this
case, the convergence rate of the error in the $L^2$-norm is
suboptimal with only half a power of $h$. Hence the non-optimality for
the non-symmetric Nitsche's method for continuous Galerkin methods is not as important as for
DG-methods (see \cite{OBB98} and \cite{GR09} for numerical evidence of
the sub-optimal behavior in this latter case).

We then show how the results may
be applied in the case of convection--diffusion equations, considering
first the Streamline--diffusion method and then outlining how the
results may be extended to the case of the Continuous interior penalty
method. 

Nitsche's method however has some stabilizing properties of
its own, in particular for outflow layers, this phenomenon was
analyzed in \cite{Sch08} and illustrated herein with a numerical example.
This makes the non-symmetric Nitsche's
method an appealing, parameter free, method for flow problems where
the system matrix is 
non-symmetric and the use of stabilized methods usually also results
in the loss of half a power of $h$. It should be noted however that
the smallest error in the $L^2$-norm is obtained with the formulation
using penalty on the boundary, as illustrated in the numerical
section. So we do not claim that the penalty
free method is the most accurate. 

We only prove the result
in the case of the imposition of boundary conditions but the extensions
of the results to the domain decomposition case of \cite{BHS03} or the
fictitious domain method of \cite{Bu11} are
straightforward using similar techniques as below. Also note that
since the main aim of the present paper is the study of weak
imposition of boundary conditions, we will assume that the reader has 
basic understanding of the techniques for analyzing stabilized finite
element methods and some arguments are only sketched.

For the sake of clarity, we first prove the main result on the pure diffusion problem and then
discuss the extension of our result to the case of convection--
diffusion problems. The paper is ended with some numerical examples.

\section{The pure diffusion problem}
Let $\Omega$ be a bounded domain in $\IR^2$, with polygonal
boundary $\partial \Omega$. Wherever $H^2$-regularity of the exact
solution is needed we
also assume that $\Omega$ is convex. Let $\{\Gamma_i\}_i$ denote the faces of
the polygonal such that $\partial \Omega = \cup_i \Gamma_i$.
The Poisson equation that we propose as a model problem is given
by
\begin{equation}\label{poisson}
\begin{array}{rcl}
- \Delta u &=& f\quad \mbox{ in } \Omega,\\
u&=&g \quad \mbox{ on } \partial \Omega,
\end{array}
\end{equation}
where $f \in L^2(\Omega)$ and $g \in H^{1/2}(\partial \Omega)$ or $g \in H^{3/2}(\partial \Omega)$.

We have the following weak formulation: find $u \in V_g$ such that
\begin{equation}
a(u, v) = (f,v) _\Omega, \quad \forall v \in V_0,
\end{equation}
where $(x,y)_\Omega$ denotes the $L^2$-scalar product over $\Omega$, 
\[
V_g := \{v \in H^1(\Omega) : v\vert_{\partial \Omega} = g \}
\]
and 
\[
a(u,v) := ( \nabla u,\nabla v)_{\Omega}.
\]
This problem is well-posed by the Lax-Milgram's lemma, using the
standard arguments to account for non-homogeneous boundary
conditions. 
The $H^1$-stability, $\|u\|_{H^1(\Omega)}
\leq C_{R1} (\|f\|+\|g\|_{H^{1/2}(\partial \Omega)})$ holds and under the
assumptions on $\Omega$, $f$ and $g$ there holds  $\|u\|_{H^2(\Omega)}
\leq C_{R2} (\|f\|+\|g\|_{H^{3/2}(\partial \Omega)})$. Here we let
$\|x\| := \|x\|_{L^2(\Omega)}$.
Below $C$ will be used as a generic constant that may change at each
occasion, is independent on $h$,
but not necessarily of the local mesh geometry. We will also use the
notation $a \lesssim b$ for $a \leq C b$.
\section{The finite element formulation}
\label{S:FEM}
Let $\{\mathcal{T}_h\}$ denote a family of quasi uniform and shape regular
triangulations fitted to $\Omega$, indexed by the mesh-parameter
$h$. (It is straightforward to lift the quasi uniformity assumption, at
the expense of some standard technicalities and readability.) The triangles 
of $\mathcal{T}_h$ will be denoted $K$ and their diameter $h_K :=
\mbox{diam}(K)$. The interior of a set $P$
will be denoted $\overset{\circ}{P}$. For a given $\mathcal{T}_h$ the mesh-parameter is
determined by $h := \max_{K \in \mathcal{T}_h} h_K$. Shape regularity
is expressed by the existence of a constant $c_\rho \in \mathbb{R}$
for the family of triangulations such that, with $\rho_K$ the radius of the
largest ball inscribed in an element $K$, there holds,
\[
\frac{h_K}{\rho_K} \leq c_\rho, \forall K \in \mathcal{T}_h.
\]
For technical reasons, and to avoid the treatment of special cases, we
assume that for all $i$, $\Gamma_i$ contains no less than five element faces.

We introduce the standard finite element space of continuous piece wise
polynomial functions
\[
V^k_h := \{ v_h \in H^1(\Omega): v_h \vert_{K} \in \mathbb{P}_k(K),\quad \forall
K \in \mathcal{T}_h\},\, k \ge 1,
\]
where $\mathbb{P}_k(K)$ denotes the space of polynomials of degree less than or
equal to $k$ on the element $K$.
The finite element formulation that we consider then takes the form,
find $u_h \in V^k_h$ such that
\begin{equation}\label{FEM}
a_h(u_h,v_h) = (f,v_h)_{\Omega} + \left<g, \nabla v_h \cdot n
\right>_{\partial \Omega}\quad \forall v_h \in V_h^k,
\end{equation}
where $\left<x,y\right>_{\partial \Omega}$ denotes the $L^2$-scalar
product over the boundary of $\Omega$
and
\begin{equation}\label{disc_biform}
a_h(u_h,v_h) := a(u_h,v_h) - \left<\nabla u_h \cdot n, v_h \right>_{\partial \Omega} + \left<u_h,\nabla v_h \cdot n \right>_{\partial \Omega}.
\end{equation}
Note that in the classical non-symmetric Nitsche's method we also add
a penalty term of the form
\begin{equation}\label{penalty}
\sum_K \left<\gamma   h_K^{-1} u_h, v_h \right>_{\partial \Omega
  \cap \partial K}
\end{equation}
and modify the second term of the right hand side accordingly
\[
\sum_K \left<g, \gamma  h_K^{-1} v_h + \nabla v_h \cdot n \right>_{\partial \Omega \cap \partial K}.
\]
The key observation of the present work is that the penalty parameter
$\gamma$ may be chosen to be zero without loss of neither stability
nor accuracy.

Inserting the exact solution $u$ into the fomulation \eqref{FEM} and
integrating by parts immediately leads to the following consistency relation.
\begin{lemma}\label{galortho}
If $u$ is the solution of \eqref{poisson} and $u_h$ is the solution of
\eqref{FEM} then there holds
\[
a_h(u - u_h,v_h) = 0.
\]
\end{lemma}
For future reference we here recall the classical trace and inverse
inequalities satisfied by the spaces $V_h^k$.
\begin{lemma}(Trace inequality)\label{trace_ineq}
There exists $C_T \in \mathbb{R}$ such that
for all $v_h \in \mathbb{P}_k(K)$ and for all $K \in \mathcal{T}_h$ there holds
\[
\|v_h\|_{L^2(\partial K)} \leq C_T(h_K^{-\frac12} \|v_h\|_{L^2(K)} +
h_K^{\frac12} \|\nabla v_h\|_{L^2(K)}).
\]
\end{lemma}
\begin{lemma}(Inverse inequality)\label{inverse_ineq}
There exists $C_I \in \mathbb{R}$ such that
for all $v_h \in \mathbb{P}_k(K)$ and for all $K \in \mathcal{T}_h$ there holds
\[
\|\nabla v_h\|_{L^2(K)} \leq C_I h_K^{-1} \|v_h\|_{L^2(K)}.
\]
\end{lemma}
\section{Stability}\label{stability}
The non-symmetric Nitsche's method is positive and testing with $v_h =
u_h$
immediately gives control of the $H^1$-seminorm of $u_h$. In  order
for the formulation to be well-posed this is not sufficient. Indeed
well-posedness is a consequence of the Poincar\'e inequality that
holds provided we have sufficient control of the trace of $u_h$ on
$\partial \Omega$. This is the role of the penalty term \eqref{penalty}, it ensures
that the following Poincar\'e inequality is satisfied
\[
\|u_h\| \leq C_P \|u_h\|_{1,h} , \mbox{ where } \|u_h\|_{1,h}^2:=\|\nabla u_h\|^2 + \|u_h\|^2_{\frac12,h,\partial \Omega} 
\]
with
$$
\|u_h\|^2_{\frac12,h,\partial \Omega}:= \sum_K \left<h_K^{-1} u_h,u_h
\right>_{\partial \Omega \cap \partial K}.
$$
Since we have omitted the penalty term, boundary control
of $u_h$ is not an
immediate consequence of testing with $v_h = u_h$. What we will show
below
is that control of the boundary term can be recovered by proving an
inf-sup condition. Indeed the non-symmetric Nitsche's method can be
interpreted as a Lagrange multiplier method where the Lagrange
multiplier $\lambda_h$ has been replaced by the normal gradient of the
solution: $\nabla u_h \cdot n$. This interpretation of the Nitsche's
method was originally proposed in \cite{Sten95}, however without considering
the inf-sup condition. In the DG-framework it was considered in
\cite{BS10}, where equivalence was shown between a certain
Lagrange-multiplier method and a certain DG-method.
When Lagrange-multipliers are used to impose continuity,
the system has a saddle point structure and the inf-sup condition is
the standard way of proving well-posedness. Here we will follow a
similar procedure, the only difference is that the solution space
and the multiplier space are strongly coupled, since the latter consists
simply of the normal gradients of the former. A key result is given in the
following lemma where we construct a function in the test space that
will allow us to control certain averages of the solution on the
boundary. 
To this end regroup the boundary elements, i.e. the elements with
either a face or a vertex on the boundary, in (closed) patches $P_j$, with
boundary $\partial P_j$,
$j=1...N_P$. Let $F_j := \partial P_j \cap \partial \Omega$.
We assume that the $P_j$ are designed such that each $F_j$ has at
least four inner nodes (this is strictly necessary only if both end
vertices of $P_j$ belong to corner elements with all their vertices on
the boundary). Under our assumptions on the mesh, every $\Gamma_i$
contains at least one patch $P_j$ and there
exists $c_1,c_2$ such that for all $j$
\begin{equation} \label{Pcond}
c_1 h \leq \mbox{meas}(F_j) \leq c_2h.
\end{equation}
The average value of a function $v$ over $F_j$ will be denoted by
$\bar v^j$.
\begin{lemma}\label{spec_func}
For any given vector $(r_j)_{j=1}^{N_P} \in \mathbb{R}^{N_P}$
there exists $\varphi_r \in V_h^1$ such that for all $1 \leq j \leq N_P$
there holds
\begin{equation}\label{constraint_vp}
\mbox{meas}(F_j)^{-1} \int_{F_j} \nabla \varphi_r \cdot n
~\mbox{d}s = r_j
\end{equation}
and
\begin{equation}\label{stability_vp}
\|\varphi_r\|_{1,h} \lesssim \left(\sum_{j=1}^{N_P}
\|h^{\frac12} r_j\|^2_{L^2(F_j)}\right)^{1/2}.
\end{equation}
\end{lemma}
\begin{proof}
We first construct a function $\tilde \varphi_j$ taking the value $1$
in the interior nodes of $\partial \Omega \cap \partial P_j$ and zero elsewhere.
Fix $j$ and let $\tilde \varphi_j \in V_h^1$ be defined, in each vertex $x_i \in
\mathcal{T}_h$, by
\[
\tilde \varphi_j(x_i) = \left\{ \begin{array}{ll}
0 & \mbox{ for } x_i \in K \mbox{ such that $K$ has three vertices on
  $\partial \Omega$};  \\
0 &  \mbox{ for } x_i  \in \Omega \setminus \overset{\circ}{P_j};\\
1 & \mbox{ for } x_i \in \overset{\circ}{F_j}.
\end{array}
\right.
\]
Let 
\[
\Xi_j := \mbox{meas}(F_j)^{-1} \int_{F_j} \nabla \tilde \varphi_j \cdot
n ~ \mbox{d}s
\]
and define the normalised function $\varphi_j$ by
\[
\varphi_j := \Xi_j^{-1} \tilde \varphi_j.
\]
This quantity is well defined thanks to the following lower bound
that holds uniformly in $j$ and $h$
\[
C_\Xi \leq \Xi_j h.
\]
The constant $C_\Xi$ only depends on the local geometry of the patches
$P_j$. 
By definition there holds
\begin{equation}\label{average_vp}
 \mbox{meas}(F_j)^{-1} \int_{F_j} \nabla \varphi_j \cdot
n ~ \mbox{d}s = 1
\end{equation}
and using the standard inverse inequality (Lemma \ref{inverse_ineq})
\begin{equation}\label{inverse_vp}
\|\nabla \varphi_j\| \lesssim C_I h^{-1} \Xi_j^{-1} \|\tilde
\varphi_j\|_{L^2(P_j)} \lesssim C_I h^{-1} \Xi_j^{-1}
\mbox{meas}(P_j)^{1/2}\lesssim C_I C^{-1}_\Xi  h.
\end{equation}
Now defining 
\[
\varphi_r :=
\sum_{j=1}^{N_P} r_j \varphi_j
\]
we immediately see that condition \eqref{constraint_vp} is satisfied
by equation \eqref{average_vp}. The upper bound
\eqref{stability_vp} follows from
\eqref{inverse_vp}, the relation \eqref{Pcond}
and using that
\begin{multline*}
\|\varphi_r\|_{\frac12,h,\partial \Omega}^2 := \sum_{j=1}^{N_P}
\|h^{-\frac12} r_j \varphi_j\|^2_{L^2(F_j)} \\
\lesssim
\sum_{j=1}^{N_P} h^{-1} r_j^2 \Xi^{-2}_j\| \tilde
\varphi_j\|^2_{L^2(F_j)} \lesssim C_{\Xi}^{-2}\sum_{j=1}^{N_P} \|h^{\frac12} r_j \|^2_{L^2(F_j)}.
\end{multline*}
\end{proof}

With the help of this technical lemma it is straightforward to prove
the inf-sup condition for the formulation \eqref{FEM}.
\begin{theorem}\label{infsup}
There exists $c_s>0$ such that for all functions $v_h \in
V_h^k$ there holds
\[
c_s \|v_h\|_{1,h} \leq \sup_{w_h \in V_h^k} \frac{a_h(v_h,w_h)}{\|w_h\|_{1,h}}.
\]
\end{theorem}
\begin{proof}
Recall that 
\[
a_h(v_h,w_h) = (\nabla v_h,\nabla w_h)_\Omega - \left<\nabla v_h \cdot n, w_h \right>_{\partial \Omega} + \left<v_h,\nabla w_h \cdot n \right>_{\partial \Omega}.
\]
Taking $w_h = v_h$ gives
\[
a_h(v_h,v_h)= \|\nabla v_h\|^2.
\]
To recover control over the boundary integral we let
\begin{equation}\label{rchoice}
r_j = h^{-1} \bar v^j  := h^{-1} \mbox{meas}(F_j)^{-1}\int_{F_j} v_h ~ \mbox{d}s
\end{equation}
in the construction of $\varphi_r$ in Lemma \ref{spec_func} and note that
\[
\left<v_h,\nabla \varphi_r \cdot n \right>_{\partial \Omega} =
\sum_{j=1}^{N_P} \left(\|h^{-1/2}\bar v^j\|^2_{L^2(F_j)}  + \left< (v_h - \bar v^j),\nabla \varphi_r \cdot n \right>_{F_j} \right).
\]
Using standard approximation, 
\begin{equation}\label{L2approx}
\|v_h - \bar v^j \|_{L^2(F_j)} \lesssim h \|\nabla v_h \times
n\|_{L^2(F_j)},
\end{equation}
 and by the trace and inverse inequalities of Lemma
\ref{trace_ineq} and Lemma \ref{inverse_ineq} we have
\[
\left< (v_h - \bar v^j),\nabla \varphi_r \cdot n \right>_{F_j} \lesssim C_T^2 (1+C_I)  \|\nabla v_h\|_{L^2(P_j)} \|\nabla \varphi_r\|_{L^2(P_j)}.
\]
Moreover since by Cauchy-Schwarz inequality and the trace inequality,
\[
|(\nabla v_h,\nabla w_h)_\Omega - \left<\nabla v_h \cdot n, w_h
\right>_{\partial \Omega}| \lesssim \|\nabla v_h\| \|w_h\|_{1,h}
\]
we deduce using the stability \eqref{stability_vp} that
\begin{multline*}
a_h(v_h,\varphi_r) \ge \sum_{j=1}^{N_P} \|h^{-1/2} \bar v^j\|^2_{L^2(F_j)}
- C \|\nabla v_h\| \|\varphi_r\|_{1,h} \\ \geq \sum_{j=1}^{N_P}
\|h^{-1/2} \bar v^j\|^2_{L^2(F_j)}
-
C_s \|\nabla v_h\| \left(\sum_{j=1}^{N_P} \|h^{-1/2} \bar v^j\|^2_{L^2(F_j)}\right)^{1/2}.
\end{multline*}
We now fix $w_h = v_h + \eta \varphi_r$ and note that
\begin{multline}\label{positiv_is}
a_h(v_h,w_h) \ge \|\nabla v_h\|^2 + \eta \sum_{j=1}^{N_P} \|h^{-1/2} \bar
  v^j\|^2_{L^2(F_j)} \\ -
C_s \|\nabla v_h\| \eta \left(\sum_{j=1}^{N_P} \|h^{-1/2} \bar
  v^j\|^2_{L^2(F_j)}\right)^{1/2}\\
\ge (1-\epsilon) \|\nabla v_h\|^2 + \eta (1 - C_s^2 \eta/(4\epsilon)) \sum_{j=1}^{N_P} \|h^{-1/2} \bar
  v^j\|^2_{L^2(F_j)}.
\end{multline}
It follows, using once again the approximation properties of the
$L^2$-projection on the piece wise constants \eqref{L2approx}, that for any $\epsilon < 1$ we may take $\eta$ sufficiently
small so that there exists $c_{\eta,\epsilon}$ such that
\[
c_{\eta,\epsilon} \|v_h\|_{1,h}^2 \leq C c_{\eta,\epsilon} \left(\|\nabla
v_h\|^2+\sum_{j=1}^{N_P} \|h^{-1/2} \bar
  v^j\|^2_{L^2(F_j)}\right)\leq  a_h(v_h,w_h).
\]
We may conclude by noting that by \eqref{stability_vp}, our choice
of $r_j$ and the stability of the $L^2$-projection on piece wise
constants there holds
\begin{equation}\label{whstab}
\|w_h\|_{1,h} \leq \|v_h\|_{1,h} + \eta \|\varphi_r\|_{1,h} \leq C_\eta  \|v_h\|_{1,h}.
\end{equation}
\end{proof}
\section{A priori error estimates}
The stability estimate proved in the previous section together with
the Galerkin orthogonality of Lemma \ref{galortho} leads to error
estimates in the $\|\cdot\|_{1,h}$ norm in a straightforward
manner. First we will prove an auxiliary lemma for the continuity of
$a_h(\cdot,\cdot)$. To this end we introduce the norm
\[
\|u\|_* := \|u\|_{1,h} + \|h^{\frac12} \nabla u \cdot n\|_{L^2(\partial \Omega)}.
\]
\begin{lemma}\label{continuity}
Let $u \in H^2(\Omega) + V_h^k$ and $v_h \in V_h^k$.
Then the bilinear form $a_h(\cdot,\cdot)$ defined by \eqref{disc_biform}
satisfies
\[
a_h(u,v_h) \leq C \|u\|_* \|v_h\|_{1,h}.
\]
\end{lemma}
\begin{proof}
The result is immediate by application of the Cauchy-Schwarz
inequality and the trace inequality of Lemma \ref{trace_ineq}.
\end{proof}
\begin{proposition}\label{apriori_diff}
Let $u \in H^{k+1}(\Omega)$ be the solution of
\eqref{poisson} and $u_h$ the solution of \eqref{FEM}.
Then there holds
\[
\|u-u_h\|_{1,h} \leq C h^k |u|_{H^{k+1}(\Omega)}.
\]
\end{proposition}
\begin{proof}
Let $i_{\tt SZ}^k u$ denote the Scott-Zhang interpolant of $u$
\cite{SZ90}. Using the approximation properties of the interpolant it
is
straightforward to show that $$\|u - i_{\tt SZ}^k u\|_{1,h}+\|u - i_{\tt SZ}^k u\|_* \lesssim h^k |u|_{H^{k+1}(\Omega)}.$$
We therefore use the triangle inequality to obtain
\[
\|u-u_h\|_{1,h}  \leq \|u-i_{\tt SZ}^k u\|_{1,h} +\|u_h-i_{\tt SZ}^k u\|_{1,h},
\]
where only the second term needs to be bounded.
To this end we apply the result of Theorem \ref{infsup} followed by the
consistency of Lemma \ref{galortho} 
\[
c_s \|u_h-i_{\tt SZ}^k u\|_{1,h}\leq \sup_{w_h \in V_h^k} \frac{a_h(u_h-i_{\tt SZ}^k u,w_h)}{\|w_h\|_{1,h}} = \sup_{w_h \in V_h^k} \frac{a_h(u-i_{\tt SZ}^k u,w_h)}{\|w_h\|_{1,h}}.
\]
By the continuity of Lemma \ref{continuity} and the approximation
properties of $i_{\tt SZ}^k u$ we conclude
\[
c_s \|u_h-i_{\tt SZ}^k u\|_{1,h}\lesssim \|u - i_{\tt SZ}^k u\|_* \lesssim h^k |u|_{H^{k+1}(\Omega)}.
\]
\end{proof}\\

For DG-methods it is well-known that the non-symmetric version may
suffer from suboptimality in the convergence of the error in the $L^2$-norm due to the lack of adjoint
consistency.
This is true also for the non-symmetric Nitsche's method, however
since the method is used on the scale of the domain and not of the
element the suboptimality may be reduced to $h^{\frac12}$ as we prove
below.
\begin{proposition}
Let $u \in H^{k+1}(\Omega)$ be the solution of
\eqref{poisson} and $u_h$ the solution of \eqref{FEM}.
Then 
\[
\|u-u_h\| \leq C h^{k+\frac12} |u|_{H^{k+1}(\Omega)}.
\]
\end{proposition}
\begin{proof}
Let $z$ satisfy the adjoint problem
\[
\left\{\begin{array}{rcll}
-\Delta z &=& u - u_h & \mbox{in } \Omega,\\
z &=& 0 & \mbox{on } \partial \Omega.
\end{array}\right.
\]
Under the assumptions on $\Omega$ we know that $\|z\|_{H^2(\Omega)}
\leq C_{R2} \|u - u_h\|$.
It follows that
\begin{multline*}
\|u - u_h\|^2 = (u- u_h,-\Delta z)_{\Omega} = (\nabla (u-u_h), \nabla z)
_{\Omega} - \left< u- u_h,\nabla z \cdot n\right>_{\partial \Omega}\\
 = a_h(u-u_h,z) + 2 \left< u- u_h,\nabla z \cdot n\right>_{\partial \Omega}.
\end{multline*}
By Lemma \ref{galortho} and a continuity argument similar to that of Lemma
\ref{continuity}, using that $(z - i^1_{\tt SZ} z)\vert_{\partial \Omega} \equiv
0$, it follows that 
\begin{multline}\label{bilinform}
a_h(u-u_h,z) = a_h(u-u_h,z - i_{\tt SZ}^1 z) \\
= (\nabla (u-u_h),\nabla (z - i_{\tt SZ}^1
z))_{\Omega} -  \left< u- u_h,\nabla (z - i_{\tt SZ}^1 z) \cdot
  n\right>_{\partial \Omega} \\
\lesssim  \|u-u_h\|_{1,h} \|z-i_{\tt SZ}^1 z\|_*\\
\lesssim h \|u-u_h\|_{1,h} |z|_{H^2(\Omega)}.
\end{multline}
We also have, using the following global trace inequality 
$$
\|\nabla z \cdot n\|_{L^2(\partial \Omega)} \lesssim \|z\|_{H^2(\Omega)},
$$
that 
\begin{equation}\label{nonconst}
|\left< u- u_h,\nabla z \cdot n\right>_{\partial \Omega}| \lesssim h^{1/2}
\|u- u_h\|_{\frac12,h,\partial \Omega} \|z\|_{H^2(\Omega)}.
\end{equation}
Collecting the inequalities \eqref{bilinform} and \eqref{nonconst} we
arrive at the estimate
\[
\|u-u_h\|^2 \lesssim (h + h^{1/2}) h^k |u|_{H^{k+1}(\Omega)} \|z\|_{H^2(\Omega)}
\] 
and we conclude by applying the regularity estimate $\|z\|_{H^2(\Omega)}
\leq C_{R2} \|u - u_h\|$.
\end{proof}
\section{The convection--diffusion problem}
Since the method we discuss leads to a non-symmetric
system matrix the main interest of the method is for solving flow
problems where an advection term makes the problem non-symmetric
anyway. Note that there appears to be no analysis that is robust with
respect to the P\'eclet number, even in the case of the non-symmetric 
discontinuous Galerkin method.

We will therefore now show how the above analysis can be
extended to the case of convection--diffusion equations yielding
optimal stability and accuracy both in the convection and the
diffusion dominated regime. We will consider the following
convection--diffusion--reaction equation:
\begin{equation}\label{conv_diff}
\sigma u + \beta\cdot \nabla u - \varepsilon \Delta u = f \mbox{ in } \Omega,
\end{equation}
and homogeneous Dirichlet boundary conditions. We
assume that  $\beta \in [W^1_\infty(\Omega)]^2$, $\sigma \in \mathbb{R}$,
$$
\sigma - \frac12 \nabla \cdot
\beta \ge c_\sigma \ge 0
$$ and $\varepsilon \in
\mathbb{R}^+$.
 In this case the formulation writes: find $u_h \in V_h$ such that
\begin{multline}\label{convdiffFEM}
A_h(u_h,v_h) := (\sigma u_h + \beta \cdot \nabla u_h, v_h)_\Omega - \left< \beta \cdot n,
  u_h,v_h\right>_{\partial \Omega^-}\\ + \varepsilon a_h(u_h,v_h) =
(f,v_h)_\Omega,\quad \forall v_h \in V_h
\end{multline}
where $\partial \Omega^\pm := \{x \in \partial \Omega : \pm \beta
\cdot n > 0\}$.
First note that the positivity of the form now writes
\begin{equation}\label{pos2}
A_h(u_h,u_h) \ge \frac12 \| |\beta \cdot n|^{\frac12} u_h\|^2_{\partial \Omega}
+ \|\varepsilon^{\frac12} \nabla u_h\|^2,
\end{equation}
hence provided $|\beta \cdot n|>0$ on some portion of the boundary 
with non-zero measure the matrix is invertible. In the following we
assume that this is the case, but we do not assume that $|\beta \cdot
n|>0$ everywhere on $\partial \Omega$.
To prove optimal error
estimates in general we require stronger stability results of the type
proved above to hold. It appears difficult to prove these stronger
results
independently of the flow regime. Indeed it is convenient
to characterize the flow using the local P\'eclet number:
\[
Pe := \frac{|\beta| h}{\varepsilon}.
\] 
If $Pe<1$ the flow is said to be diffusion dominated and if $Pe>1$ we
say that it is convection dominated. We will now treat these two cases
separately.

In view of the equality \eqref{pos2} we introduce the following
strengthened norm
\[
\|v_h\|^2_{1,h,\beta}:= \varepsilon \|v_h\|_{1,h}^2 + \frac12 \| |\beta \cdot n|^{\frac12} v_h\|^2_{\partial \Omega}.
\]
This norm is suitable in the diffusion dominated regime, but will be
modified by the introduction of stabilization when the convection dominated
regime is considered.
\subsection{Diffusion dominated regime $Pe<1$}
In this case we may prove an inf-sup condition similar to that of
Theorem \ref{infsup}. For simplicity we assume that $\sigma=0$.
\begin{proposition}(Inf-sup for convection--diffusion, $Pe<1$.)
\label{diff_conv_infsup}
For all functions $v_h \in
V_h^k$ there holds
\begin{equation}\label{infsup2}
c_s \|v_h\|_{1,h,\beta} \leq  \sup_{w_h \in V_h^k} \frac{A_h(v_h,w_h)}{\|w_h\|_{1,h,\beta}}.
\end{equation}
\end{proposition}
Clearly, compared to the proof of Theorem \ref{infsup} we only need
to show how to handle the term
\[
(\beta \cdot \nabla v_h,\varphi_r)_\Omega - \left< \beta \cdot n\,
  v_h,\varphi_r\right>_{\partial \Omega^-}.
\]
The necessary bound on this term is given in the following Lemma.
\begin{lemma}\label{conv_diff_infsup}
Let $\varphi_r$ be the function of Lemma \ref{spec_func} with $r$ chosen as in
\eqref{rchoice}.
Then for $Pe<1$ there holds for all $\mu>0$ 
\begin{multline*}
(\beta \cdot \nabla v_h, \eta \varphi_r)_\Omega - \left< \beta \cdot n\,
  v_h, \eta \varphi_r\right>_{\partial \Omega^-} \\
\leq \mu
(\varepsilon \| \nabla v_h\|^2 +\||\beta \cdot n|^{\frac12}
v_h\|^2_{L^2(\partial \Omega)})  + C_{\partial}^2 (4\mu)^{-1}  \eta^2  \varepsilon \| v_h\|^2_{\frac12,h,\partial \Omega}.
\end{multline*}
\end{lemma}
\begin{proof}
Let 
\[
(\beta \cdot \nabla v_h, \eta \varphi_r)_\Omega - \left< \beta \cdot n\,
  v_h, \eta \varphi_r\right>_{\partial \Omega^-}  = T_1 + T_2.
\]
By the definition of the P\'eclet number and the Cauchy-Schwarz inequality, we have
\[
T_1 \leq Pe \varepsilon^{\frac12} \|\nabla v_h\| \eta
\varepsilon^{\frac12} \|h^{-1} \varphi_r\|.
\]
From the construction of $\varphi_r$,  a scaling argument, the
stability \eqref{stability_vp}  and the choice of $r$ \eqref{rchoice} we deduce that
\[
 \|h^{-1} \varphi_r\| \lesssim \|\nabla \varphi_r\| \leq C_\partial  \|v_h\|_{\frac12,h,\partial \Omega}.
\]
Using the arithmetic-geometric inequality we have
\[
T_1 \leq \mu \varepsilon\|\nabla v_h\|^2 + C_\partial^2 (4\mu)^{-1} Pe^2 \eta^2
\varepsilon  \|v_h\|_{\frac12,h,\partial \Omega}^2.
\]
For $T_2$ we have using a Cauchy-Schwarz inequality, the definition of
the P\'eclet number and the stability \eqref{stability_vp}
\begin{multline*}
T_2 \leq \||\beta \cdot n|^{\frac12} v_h\|_{L^2(\partial \Omega)}
Pe^{\frac12} \eta \varepsilon^{\frac12} \|\varphi_r\|_{\frac12,h,\partial
  \Omega} 
\leq C_\partial  \||\beta \cdot n|^{\frac12} v_h\|_{L^2(\partial \Omega)} \eta
\varepsilon^{\frac12} \|v_h\|_{\frac12,h,\partial
  \Omega}.
\end{multline*}
We apply the arithmetic-geometric inequality once again to conclude.
\end{proof}
\begin{proof} (Proposition \ref{diff_conv_infsup})
The inf-sup stability \eqref{infsup2} now follows by taking $w_h:=
v_h+\eta \varphi_r $ and proceeding as in
equation \eqref{positiv_is} using \eqref{pos2} and Lemma \ref{conv_diff_infsup} in
the following fashion
\begin{multline*}
A_h(v_h,v_h+\eta \varphi_r) \ge (1 - \epsilon - \mu) \varepsilon \|
\nabla v_h\|^2 + (\frac12 - \mu) \||\beta \cdot n|^{\frac12}
v_h\|^2_{L^2(\Omega)} \\
+ \eta (1 - C_s^2 \eta/(4 \epsilon) - C_\partial^2 \eta/(4 \mu)) \varepsilon \| v_h\|^2_{\frac12,h,\partial \Omega}.
\end{multline*}
We may now choose $\epsilon=1/4$ and $\mu=1/4$ and then $\eta$ small
enough so that positivity is ensured. Then
\[
A_h(v_h,v_h+\eta \varphi_r) \ge C_\eta \|v_h\|_{1,h,\beta}^2.
\]
 We conclude as in Theorem \ref{infsup}, but now using the norm $\|\cdot\|_{1,h,\beta}$,
\[\begin{split}
\|w_h\|_{1,h,\beta} \leq \|v_h\|_{1,h,\beta} + \eta \|\varphi_r\|_{1,h,\beta}
\leq \|v_h\|_{1,h,\beta}+ \eta C \|v_h\|_{1,h,\beta} + \eta \||\beta \cdot
n|^\frac12 \varphi_r\|_{L^2(\partial \Omega)}\\
\leq C \|v_h\|_{1,h,\beta} + Pe^{\frac12}\eta  \varepsilon^{\frac12}
\|\varphi_r\|_{1,h} \leq C_{Pe, \eta} \|v_h\|_{1,h,\beta}.
\end{split}
\]
\end{proof}

Proceeding as in Proposition \ref{apriori_diff}, this leads to optimal a priori
estimates in the norm $\|\cdot\|_{1,h}$ for $Pe<1$.
\begin{proposition}
Let $u \in H^{k+1}(\Omega)$ be the solution of
\eqref{conv_diff} and $u_h$ the solution of \eqref{convdiffFEM} and assume that $Pe<1$.
Then 
\[
\|u-u_h\|_{1,h} \leq C h^k |u|_{H^{k+1}(\Omega)}.
\]
\end{proposition}
\begin{proof}
As in the proof of Proposition \ref{apriori_diff} we arrive at the following
representation of the discrete error
\[
c_s \|u_h-i_{\tt SZ}^k u\|_{1,h,\beta}\leq \sup_{w_h \in V_h^k} \frac{A_h(u_h-i_{\tt SZ}^k u,w_h)}{\|w_h\|_{1,h,\beta}} = \sup_{w_h \in V_h^k} \frac{A_h(u-i_{\tt SZ}^k u,w_h)}{\|w_h\|_{1,h,\beta}}.
\]
By the continuity of Lemma \ref{continuity}  and an integration by
parts in the convective term we obtain
\begin{multline*}
A_h(u_h-i_{\tt SZ}^k u,w_h)\lesssim \varepsilon \|u - i_{\tt
    SZ}^k u\|_* \|w_h\|_{1,h}\\
+ (u - i_{\tt SZ}^k u,\beta \cdot \nabla w_h)_\Omega +
  \left<\beta \cdot n (u - i_{\tt SZ}^k u), w_h \right>_{\partial
    \Omega^+})\\
\lesssim \varepsilon ( \|u - i_{\tt
    SZ}^k u\|_* + Pe \|h^{-1} (u - i_{\tt SZ}^k u)\| + Pe \|u - i_{\tt SZ}^k
  u\|_{\frac12,h,\partial \Omega}) \|w_h\|_{1,h,\beta}.
\end{multline*}
As a consequence
\begin{multline*}
\varepsilon \|u_h-i_{\tt SZ}^k u\|_{1,h} \leq \|u_h-i_{\tt SZ}^k
u\|_{1,h,\beta}\\
\lesssim c_s^{-1} \varepsilon (\|u - i_{\tt
    SZ}^k u\|_* + Pe \|h^{-1} (u - i_{\tt SZ}^k u)\| + Pe \|u - i_{\tt SZ}^k
  u\|_{\frac12,h,\partial \Omega}).
\end{multline*}
The claim follows by dividing through by $\varepsilon$, using
approximation and the assumption $Pe<1$.
\end{proof}
\subsection{Convection dominated regime: the Streamline--diffusion mehod}
In the convection dominated regime, when $Pe>1$, we need to add some stabilization
in order to obtain a robust scheme. We will here first consider the simple case
of Streamline-diffusion (SD) stabilization and assuming $\sigma=0$. In the next section the results
will be extended to include the Continuous interior penalty (CIP) method. 

The formulation now takes the form:  find $u_h \in V^k_h$ such that
\begin{multline}\label{convdiffFEM_stab}
A_{SD}(u_h,v_h) := (\beta \cdot \nabla u_h, v_h+ \delta \beta \cdot
\nabla v_h)_\Omega - \sum_K (\varepsilon \Delta u_h, \delta \beta \cdot
\nabla v_h)_K -
\left< \beta \cdot n\,
  u_h,v_h\right>_{\partial \Omega^-}\\ + \varepsilon a_h(u_h,v_h) =
(f,v_h+ \delta \beta \cdot \nabla v_h)_\Omega,\quad \forall v_h \in V^k_h,
\end{multline}
where $\delta = \gamma_{SD} h/|\beta|$ when $Pe>1$ and $\delta = 0$
otherwise. At high P\'eclet numbers, the enhanced robustness of the stabilized method allows us
to work in the stronger norm $\tnorm{ u_h }_{h,\delta}$ defined by
\begin{equation}\label{tnorm}
\tnorm{u_h}^2_{h,\delta} :=  \|\delta^{\frac12} \beta \cdot \nabla u_h\|^2 +
\frac12 \||\beta \cdot n|^{\frac12} u_h\|_{L^2(\partial \Omega)}^2 +
\varepsilon \|\nabla u_h\|^2.
\end{equation}
We will also use the weaker form $\tnorm{u_h}^2_{h,0}$ defined by
\eqref{tnorm} with $\delta=0$
and for the
convergence analysis we introduce the norm
\[
\tnorm{ u_h }^2_* := \|\delta^{-\frac12} u_h\|^2 + \varepsilon
\|\nabla u_h\cdot n\|^2_{-\frac12,h,\partial \Omega}+\sum_K
\|\delta^{\frac12} \varepsilon \Delta u_h\|^2_{L^2(K)}+
\varepsilon
\|u_h\|^2_{\frac12,h,\partial \Omega} + \tnorm{u_h}^2_{h,\delta}.
\]
Testing the formulation \eqref{convdiffFEM_stab} with $v_h = u_h$
yields the positivity
\begin{equation}\label{pos_SD}
c\tnorm{ u_h }_{h,\delta}^2 \leq A_{SD}(u_h,u_h)
\end{equation}
in the standard way using an element wise inverse inequality to absorb
the second order term, i.e.
\[
\sum_K (\varepsilon \Delta u_h, \delta \beta \cdot
\nabla u_h)_K \leq \frac12 C^2_I \gamma_{SD} Pe^{-1/2} \| \varepsilon^{\frac12} \nabla
u_h\|^2+ \frac12 \|\delta^{\frac12} \beta \cdot \nabla u_h\|^2.
\]
Clearly for $\gamma_{SD}<1/(C^2_I)$ stability holds for $Pe>1$.

Unfortunately the norms proposed above seem too weak to allow for optimal error estimates.
Indeed, since we do not control all of $\|u_h\|_{1,h}$, for general $u \in H^2+V_h^1$, $v_h \in
V_h^1$ there does not hold $A_{SD}(u,v_h) \leq \tnorm{ u }_*\tnorm{
  v_h }_{h,\delta}$, (c.f. Lemma \ref{continuity}) unless an assumption on the boundary
velocity such as $|\beta \cdot n| h > \varepsilon$ is made.
It also appears to be difficult to obtain an inf-sup condition similar to
\eqref{infsup2} in the high P\'eclet regime.
 
We therefore use another technique to prove optimal convergence
directly. 
The idea is to construct an interpolation operator $\pi_{\partial} u$, such that the
{\em interpolation error} $u - \pi_{\partial} u$ satisfies the continuity estimate:
\begin{equation}\label{cont_SD}
A_{SD}(u - \pi_{\partial} u, v_h) \leq \tnorm{ u - \pi_{\partial} u
}_* \tnorm{ v_h }_{h,\delta}.
\end{equation}
Assume that we have an interpolation operator
$\pi_\partial:H^1(\Omega) \mapsto V_h^1$ such that the following
hypothesis are satisfied.
\begin{itemize}
\item[({\bf{H1}})] Approximation,
\begin{equation}\label{prop1}
\|\pi_\partial u - u\| + h \|\nabla (\pi_\partial u - u)\| \leq
  C h^{k+1} |u|_{H^{k+1}(\Omega)}.
\end{equation}
\item[({\bf{H2}})] Normal gradient,
\begin{equation}\label{prop2}
\int_{F_i}     \nabla (\pi_\partial u- u) \cdot n ~\mbox{d}s = 0, \quad
i=1\hdots N_P,
\end{equation}
where $F_i$ are the boundary segments introduced in Section \ref{stability}.
\end{itemize}
Under assumptions ({\bf{H1}}) and ({\bf{H2}}), we may prove the
optimal convergence of the SD-method.
\begin{proposition}
Let $u \in H^{k+1}(\Omega)$ be the solution of \eqref{conv_diff} and $u_h$ the solution of \eqref{convdiffFEM_stab}.
Assume that there exists $\pi_{\partial} u \in V_h^k$ satisfying
({\bf{H1}}) and ({\bf{H2}}). Then
\[
 \tnorm{ u - u_h}_{h,\delta} \lesssim h^{k+\frac12} (1 + Pe^{-\frac12}) |u|_{H^{k+1}(\Omega)}.
\]
\end{proposition}
\begin{proof}
It follows from the approximation properties of $\pi_\partial$ that 
\[
 \tnorm{ u - \pi_{\partial} u}_* \lesssim  \|\beta\|_{\infty}^{\frac12}
 h^{k+\frac12} (1+Pe^{-1/2}) |u|_{H^{k+1}(\Omega)}.
\]
We now need to prove the continuity \eqref{cont_SD}. Note that
\begin{multline*}
A_{SD}(u - \pi_{\partial} u, v_h) = (\delta^{\frac12}\beta \cdot \nabla (u -
\pi_{\partial} u)+\delta^{-\frac12} (u -\pi_{\partial}
u),
\delta^{\frac12} \beta \cdot \nabla v_h)_{K} \\[3mm]
- \sum_{K}  (\delta^{\frac12} \varepsilon \Delta (u - \pi_\partial u), \delta^{\frac12} \beta \cdot \nabla v_h)_{K}
+ \left< \beta \cdot n\,(u - \pi_{\partial} u),v_h\right>_{\partial
  \Omega^+} + \varepsilon a_h(u - \pi_{\partial} u,v_h)\\[3mm]
\lesssim \tnorm{ u - \pi_{\partial} u}_* \tnorm{ v_h }_{h,\delta} + \underbrace{\varepsilon a_h(u - \pi_{\partial} u,v_h)}_{I_1}.
\end{multline*}
Consider now the term $I_1$. We will prove the continuity
\begin{equation}\label{epsa}
\varepsilon a_h(u - \pi_{\partial} u,v_h) \leq \varepsilon^{\frac12} \|u -
\pi_{\partial} u\|_* \varepsilon^{\frac12}  \|v_h\|_{1,h} \leq \tnorm{u -
\pi_{\partial} u}_*\tnorm{v_h}_{h,\delta}
\end{equation}
Using Cauchy-Schwarz inequality and a
trace inequality we show the continuity of the first and last term.
\begin{multline*}
I_1 = \varepsilon (\nabla (u - \pi_{\partial} u),\nabla v_h)_\Omega -\varepsilon \left<\nabla (u -
  \pi_{\partial} u) \cdot n, v_h \right>_{\partial
  \Omega}+\varepsilon \left<\nabla v_h\cdot n, (u -
  \pi_{\partial} u)\right>_{\partial
  \Omega} \\[3mm]
\leq \varepsilon^{\frac12}  \| u - \pi_{\partial} u\|_* \tnorm{ v_h }_{h,0} -
\varepsilon \left<\nabla (u -
  \pi_{\partial} u) \cdot n, v_h \right>_{\partial
  \Omega}.
\end{multline*}
For the remaining term we must exploit the orthogonality property
\eqref{prop2} of
$\pi_\partial u$ on the boundary. Indeed by decomposing the boundary
integral on the $N_P$ subdomains $F_i$ we have, denoting by $\bar v^i_h$
the average of $v_h$ over the boundary segment $F_i$.
\begin{multline*}
\varepsilon\left<\nabla (u -
  \pi_{\partial} u) \cdot n, v_h \right>_{\partial
  \Omega} = \varepsilon \sum_{i=1}^{N_P} \left<\nabla (u -
  \pi_{\partial} u) \cdot n, v_h - \bar v^i_h\right>_{F_i} \\
\leq \varepsilon \sum_{i=1}^{N_P} \|\nabla (u -
  \pi_{\partial} u) \cdot n\|_{L^2(F_i)} \| v_h - \bar v^i_h\|_{L^2(
    F_i)} \\
\lesssim \varepsilon^{\frac12} \|\nabla (u - \pi_{\partial} u) \cdot n\|_{-\frac12,h,\partial
  \Omega} \varepsilon^{\frac12} \|\nabla v_h\|\\
\lesssim \varepsilon^{\frac12}\| u - \pi_{\partial} u \|_* \varepsilon^{\frac12} \|v_h\|_{1,h} .
\end{multline*}
Where we used the approximation properties of
the local average and a trace inequality. Collecting the above
estimates and noting that $$\varepsilon^{\frac12}\| u - \pi_{\partial}
u \|_* \leq \tnorm{u - \pi_{\partial}
u}_*,\quad \varepsilon^{\frac12} \|v_h\|_{1,h} \leq \tnorm{v_h}_{h,0}$$ concludes the proof of \eqref{cont_SD}.

 Using the positivity
\eqref{pos_SD}, and the consistency of the method we have,
setting
$e_h:= u_h - \pi_\partial u$, and using that $Pe > 1$
\begin{multline*}
\tnorm{ e_h }_{h,\delta}^2 = A_{SD}(e_h,e_h) = A_{SD}(u - \pi_\partial u,
e_h)
\lesssim \tnorm{ u - \pi_{\partial} u }_* \tnorm{ e_h }_{h,\delta} \\
\lesssim h^{k+\frac12} \|\beta\|_{\infty}^\frac12 (1 + Pe^{-\frac12}) |u|_{H^{k+1}(\Omega)} \tnorm{ e_h }_{h,\delta}.
\end{multline*}
\end{proof}\\
We end this section by the following Lemma
establishing the existence of the interpolation $\pi_\partial$ with
the required properties.
\begin{lemma}\label{proj_exist_1}
The interpolation operator $\pi_\partial:H^1(\Omega) \mapsto V_h^1$
satisfying the properties ({\bf H1}) and ({\bf H2}) exists.
\end{lemma}
\begin{proof}
Let $\pi_\partial u := i_{\tt SZ}^k u + \varphi_r$ where $\varphi_r$ is
the function of Lemma \ref{spec_func} with the $r_j$ chosen such that
\[
r_j = \overline{\nabla u \cdot n}^j -  \overline{\nabla  i_{\tt SZ}^k u
  \cdot n}^j.
\]
Clearly by construction there holds
\begin{multline*}
\int_{F_i} (\nabla \pi_\partial u \cdot n - \nabla u \cdot n ) ~
\mbox{d}s
= \int_{F_i} (\nabla  i_{\tt SZ}^k u\cdot n + \nabla \varphi_r\cdot n - \nabla u \cdot n ) ~
\mbox{d}s\\
= \int_{F_i} (\nabla  i_{\tt SZ}^k u\cdot n + r_i - \nabla u \cdot n ) ~
\mbox{d}s = 0.
\end{multline*}
To prove the approximation results we decompose the error
\[
\|u - \pi_\partial u\| \leq \|u - i_{\tt SZ}^k u\| + \| i_{\tt SZ}^k u
- \pi_\partial u\| \leq C h^{k+1} |u|_{H^{k+1}(\Omega)} + \|\varphi_r\|.
\]
Using local Poincar\'e inequalities and the stability \eqref{stability_vp} of $\varphi_r$ we get 
\begin{multline*}
\|\varphi_r\| \lesssim \|h \nabla \varphi_r\| \lesssim h^{\frac32}
\left(\sum_{i=1}^{N_P} \|r_i\|^2_{L^2(F_i)} \right)^{\frac12} \\
=   h^{\frac32}
\left(\sum_{i=1}^{N_P} \|\overline{\nabla u \cdot n}^i -  \overline{\nabla  i_{\tt SZ}^k u
  \cdot n}^i \|^2_{L^2(F_i)} \right)^{\frac12}.
\end{multline*}
Using the stability of the projection onto piece wise constants, element wise
trace inequalities and finally approximation, we conclude
\begin{multline*}
 \|\overline{\nabla u \cdot n}^i -  \overline{\nabla  i_{\tt SZ}^k u
  \cdot n}^i \|^2_{L^2(F_i)} \leq \|\nabla u \cdot n -  \nabla  i_{\tt SZ}^k u
  \cdot n\|^2_{L^2(F_i)} \\
\leq  2 C_T^2 (h^{-1} \|\nabla(u - i_{\tt
    SZ}^k u)\|^2_{L^2(P_i)} + h \sum_{K \in P_i}\|D^2 (u - i_{\tt
    SZ}^k u)\|^2_{L^2(K)} )
\lesssim h^{2k-1} |u|^2_{H^{k+1}(P_i)}
\end{multline*}
where $D^2u$  is the standard multi-index notation for all the second derivatives of $u$.
We conclude that
\[
\|\varphi_r\| \lesssim h^{\frac32}
(\sum_{i=1}^{N_P} \|\nabla u \cdot n -  \nabla  i_{\tt SZ}^k u
  \cdot n \|^2_{L^2(F_i)} )^{\frac12} \lesssim h^{k+1} |u|_{H^{k+1}(\Omega)}
\]
The estimate on the gradient is immediate by
\begin{multline*}
\|\nabla (u - \pi_\partial u)\| \leq \|\nabla (u - i_{\tt SZ}^k u)\|+
\|\nabla ( i_{\tt SZ}^k u -  \pi_\partial u)\| \\
\leq \|\nabla (u -
i_{\tt SZ}^k u)\| + C_I h^{-1} \| i_{\tt SZ}^k u -  \pi_\partial u\|
\lesssim h^k |u|_{H^{k+1}(\Omega)}.
\end{multline*}
\end{proof}
%
\subsubsection{Convection dominated regime: the Continuous Interior Penalty mehod}
In this section we will sketch how the above results extend to
symmetric stabilization methods assuming that $c_\sigma >0$. To reduce 
technicalities we also assume that $\beta \in \mathbb{R}^2$.
We give a full proof only in the case of piecewise affine
finite elements. 
Recall that the
CIP method is obtained by adding a penalty term on the jump of the
gradient over element faces to the finite element formulation
\eqref{convdiffFEM}. The formulation then writes: find $u_h \in V^k_h$ such that 
\begin{equation}\label{convdiffCIP}
A_h(u_h,v_h) + J_h(u_h,v_h) = (f,v_h)_\Omega, \quad \forall v_h \in V_h^k,
\end{equation}
where
\[
J_h(u_h,v_h):= \gamma_{CIP}\sum_{K \in \mathcal{T}_h} \sum_{F \in \partial K
  \setminus \partial \Omega} \int_F h_F^2 |\beta \cdot n_F| [\nabla
u_h \cdot n_F][\nabla
v_h \cdot n_F] ~\mbox{d}s,
\]
with $[x]$ denoting the jump of the quantity $x$ over the face $F$ and
$n_F$ the normal to $F$, the orientation is arbitrary but fixed in
both cases.

The analysis once again
 depends on the construction of a special interpolant $\pi_{CIP} u 
\in V_h^k$. This time $\pi_{CIP} u$ must satisfy both the optimal approximation error
estimates of \eqref{prop1}, the property \eqref{prop2} on the normal
gradient, and the additional design condition:
\begin{align}\label{projection}
(u - \pi_{CIP} u, \beta \cdot \nabla v_h) &\leq \|h^{-\frac12} |\beta|^{\frac12}
(u - \pi_{CIP} u)\|  \gamma_{CIP}^{-
\frac12} J_h(v_h,v_h)^{\frac12},\, \forall v_h \in V_h^{k}.
\end{align}
Once such an interpolant has been proven to exist, the technique of
\cite{Bu05}, combined
with the analysis above, may be
used to prove quasi-optimal $L^2$-convergence for $c_\sigma>0$.
Using a similarly designed interpolation operator, an inf-sup
condition can be used to prove stability and error estimates in the
norm $\tnorm{\cdot}_{h,\delta}$ following \cite{BH04, BE06}.
Here we will first the error estimate in the $L^2$-norm, assuming the
existence of $\pi_{CIP} u$ and then show how to construct the interpolant in the
special case $k=1$.
\begin{proposition}
Assume that $\pi_{CIP} u \in V_h^k$ satisfying, \eqref{prop1},
\eqref{prop2} and \eqref{projection} 
exists. Let $u\in H^{k+1}(\Omega)$ be the solution to
\eqref{conv_diff}, with $c_\sigma>0$, and $u_h$ be the solution to
\eqref{convdiffCIP}. Then
$$
\|u - u_h\| \lesssim (c_\sigma)^{-1} (\sigma^{\frac12} h^\frac12 + |\beta|^{\frac12}(1 + Pe^{-\frac12})) h^{k+\frac12} |u|_{H^{k+1}(\Omega)}.
$$
\end{proposition}
\begin{proof}
Let $e_h:= u_h - \pi_{CIP} u$.
There holds with $c_\sigma>0$,
\[
c_\sigma \|e_h\|^2 +  \tnorm{e_h}^2_{h,0}+ J_h(e_h,e_h) \leq A_h(e_h,e_h)+J_h(e_h,e_h).
\]
By the consistency of the method we have
\[
c_\sigma \|e_h\|^2 + \tnorm{e_h}^2_{h,0} + J_h(e_h,e_h) \leq A_h(u-\pi_{CIP} u,e_h)-J_h(\pi_{CIP} u,e_h).
\]
Finally by the continuity \eqref{epsa}, that holds thanks to property \eqref{prop2}, we have
\begin{multline}
 A_h(u-\pi_{CIP} u,e_h)-J_h(\pi_{CIP} u,e_h) \\
=(\sigma (u -\pi_{CIP} u), e_h)+(u -\pi_{CIP} u, \beta \cdot
\nabla e_h)-\int_{\partial \Omega} \beta \cdot n (u - \pi_{CIP}
u) e_h ~\mbox{d}s \\
+ \varepsilon a_h(u - \pi_{CIP}
u,e_h) + J_h(\pi_{CIP} u, e_h) \\
\leq ((\sigma^{\frac12} h^{\frac12}+C |\beta|^{\frac12} \gamma_{CIP}^{-
\frac12}) \|h^{-\frac12} (u - \pi_{CIP} u)\|+\|u - \pi_{CIP} u
\|_{1,h,\beta} \\+\varepsilon^{\frac12} \|u - \pi_{CIP} u
\|_*+J_h(\pi_{CIP} u,\pi_{CIP} u)^{\frac12}) \\
\times ( \sigma^{\frac12} \|e_h\|^2 +  \|e_h\|^2_{1,h,\beta} + J_h(e_h,e_h) )^{\frac12}
\end{multline}
and we end the proof by applying approximation estimates.
\end{proof}\\
We will now prove the existence of the interpolant $\pi_{CIP} u$ in
the case of piecewise affine continuous finite element approximation. 
\begin{lemma}\label{proj_exist_2}
The
function $\pi_{CIP} u \in V_h^1$, satisfying \eqref{prop1},
\eqref{prop2} and \eqref{projection}, is well defined and satisfies the
approximation
estimate
\[
\|u - \pi_{CIP} u\|+ h \|\nabla (u - \pi_{CIP} u)\| \lesssim h^{2} |u|_{H^{2}(\Omega)}.
\]
\end{lemma}
\begin{proof}
We write $\pi_{CIP} u := \pi_h u +
\varphi_{CIP}$ where $\pi_h u$ denotes the $L^2$-projection on $V_h^1$
and $\varphi_{CIP} \in V_h^1$ is a function
defined on patches $P_i$ that satisfy the inequalities \eqref{constraint_vp} and
\eqref{stability_vp}, but also has the property
\[
\int_{P_i} \varphi_{CIP} ~\mbox{d}x = 0, \, i=1,\hdots,N_P.
\]
Clearly for this to hold we must modify the definition of the patches
on the faces $F_i$ to include interior nodes in the domain. For simplicity we 
assume that any element containing a node that connects to two nodes in the
boundary segment $\bar F_i$ (through edges that may be associated to
other elements) is included in the patch $P_i$ (see Figure
\ref{patch_CIP}). Define two functions $w_I$ and $w_F$ on $P_i$ (also
illustrated in Figure
\ref{patch_CIP}) such that
\[
w_I := \left\{ \begin{array}{l} 1 \mbox{ in all nodes } x \in
    \overset{\circ}{P_i}\\
0  \mbox{ in all nodes } x \in
    \Omega \setminus \overset{\circ}{P_i}\end{array} \right.,\quad w_F := \left\{ \begin{array}{l} 1 \mbox{ in all nodes } x \in
    \overset{\circ}{F_i}\\
0  \mbox{ in all nodes } x \in
    \bar \Omega \setminus \overset{\circ}{F_i}\end{array} \right..
\]
We must now show that there exists a function $\varphi_i = a w_I + b
w_F$ satisfying the two constraints
\begin{equation}\label{patch_syst}
\int_{P_i} \varphi_i~\mbox{d}x = 0,\quad \overline{\nabla
  \varphi_i \cdot n}^i = r_i.
\end{equation}
The construction of $\pi_{CIP} u$ is obtained by choosing
$
r_i = \overline{\nabla u \cdot n }^i - \overline{\nabla \pi_hu \cdot n }^i 
$
in the system \eqref{patch_syst} above and then defining $\varphi_{CIP}|_{P_i} := \varphi_i$.

To study $\varphi_i$, first map the patch $P_i$ to the reference patch $\hat P_i$.
Consider the linear system for $v:= (a,b)^T \in \mathbb{R}^2$ of the form:
\begin{multline*}
\mathcal{A} v := \left[ \begin{array}{cc} \int_{\hat P_i} \hat w_I
    ~\mbox{d}\hat x & \int_{\hat P_i} \hat w_F
    ~\mbox{d}\hat x\\
\int_{\hat F_i}\nabla \hat w_I \cdot \hat n ~\mbox{d}\hat s & \int_{\hat F_i}
\nabla \hat w_F \cdot \hat n
    ~\mbox{d}\hat s
\end{array}\right]
\left[ \begin{array}{c} a \\ b \end{array}\right] \\
= \left[\begin{array}{c} 0
    \\  \int_{\hat F_i} \nabla (\hat u -
    \pi_h\hat  u)\cdot \hat n ~\mbox{d}\hat s \end{array} \right] =:
\hat f.
\end{multline*}
We must prove that the matrix $\mathcal{A}$ is invertible, but this is immediate
noting that the two coefficients in the first line of the matrix both
are strictly positive, whereas in the second line the coefficient in
the first column is negative by construction and that in the right
column is positive. The stability estimate \eqref{stability_vp} now follows from a 
scaling argument back to the physical patch $P_i$. Indeed since the
matrix $\mathcal{A}$ is invertible we have
\[
|v| \lesssim \sup_{w \in \mathbb{R}^2} \frac{w^T \mathcal{A} v}{|w|} =
\sup_{w \in \mathbb{R}^2} \frac{w^T {\hat f}}{|w|}  = |\hat f|.
\]
By norm equivalence we have
\[
\|\hat \varphi_i\|_{\hat P_i} \lesssim \|\nabla \hat \varphi_i\|_{\hat
  P_i} \lesssim |v| \lesssim |\hat f|.
\]
After scaling back to the physical element we get
\begin{equation}\label{cip_vp_stab}
h^{-1} \|\varphi_i\|_{P_i} \lesssim \|\nabla \varphi_i\|_{P_i} \lesssim
|f| \lesssim \|h^\frac12
\nabla (u - \pi_h u) \cdot n\|_{F_i},
\end{equation}
which proves \eqref{stability_vp}.

The approximation error estimates are proven in the same way as in Lemma
\ref{proj_exist_1}.
Indeed by a similar decomposition of the error we have for this case
\[
\|u - \pi_{CIP} u\| \leq \|u - \pi_h u\| + \|\pi_h u
- \pi_{CIP} u\| \lesssim h^{2} |u|_{H^{2}(\Omega)} + \|\varphi_{CIP}\|
\]
and for $\varphi_{CIP}$ we may conclude using the proof of Lemma
\ref{proj_exist_1}, using \eqref{cip_vp_stab}.

It remains to prove the continuity \eqref{projection}. This follows from
\begin{multline*}
(u - \pi_{CIP} u, \beta \cdot \nabla v_h) = (u - \pi_h u, \beta \cdot
\nabla v_h)+\sum_{i=1}^{N_P} (\varphi_i , \beta \cdot \nabla v_h) \\
= (u - \pi_h u, \beta \cdot
\nabla v_h  - I_{CIP} \beta \cdot
\nabla v_h ) + \sum_{i=1}^{N_P} (\varphi_i , (\beta \cdot \nabla v_h-
\pi_{0,P_i} \beta \cdot \nabla v_h)).
\end{multline*}
Here $I_{CIP}$ denotes a particular quasi interpolation operator
defined using averages of $\beta \cdot \nabla v_h$ in each node (see
\cite{Bu05}) and
$\pi_{0,P_i}$ denotes the projection on piecewise constant functions
on $P_i$. Using norm equivalence on discrete spaces and mapping from
the reference patch, we observe that
\[
\|h^{\frac12}|\beta|^{-\frac12} (\beta \cdot
\nabla v_h  - I_{CIP} \beta \cdot
\nabla v_h)\|^2 \lesssim \gamma_{CIP}^{-1} J_h(v_h,v_h) 
\]
and 
\[
 \sum_{i=1}^{N_P}\|h^{\frac12}|\beta|^{-\frac12} (\beta \cdot
\nabla v_h  - \pi_{0,P_i} \beta \cdot
\nabla v_h)\|^2_{P_i} \lesssim \gamma_{CIP}^{-1} J_h(v_h,v_h).
\]
The first claim was proved in \cite{Bu05} and the second holds since $\beta \cdot
\nabla v_h$ is constant on each element.
\end{proof}
\begin{figure}
\includegraphics[width=6cm]{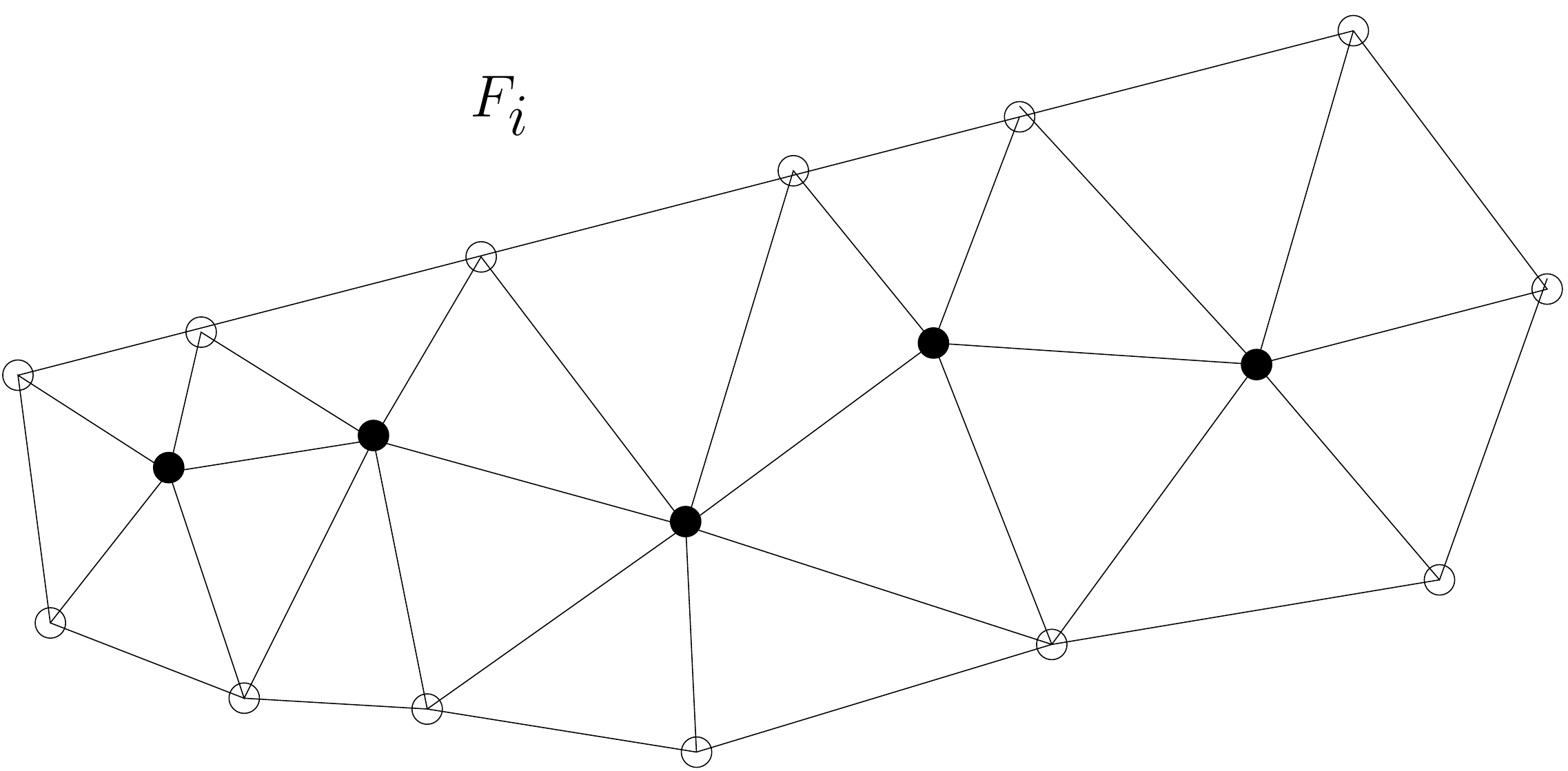}
\includegraphics[width=6cm]{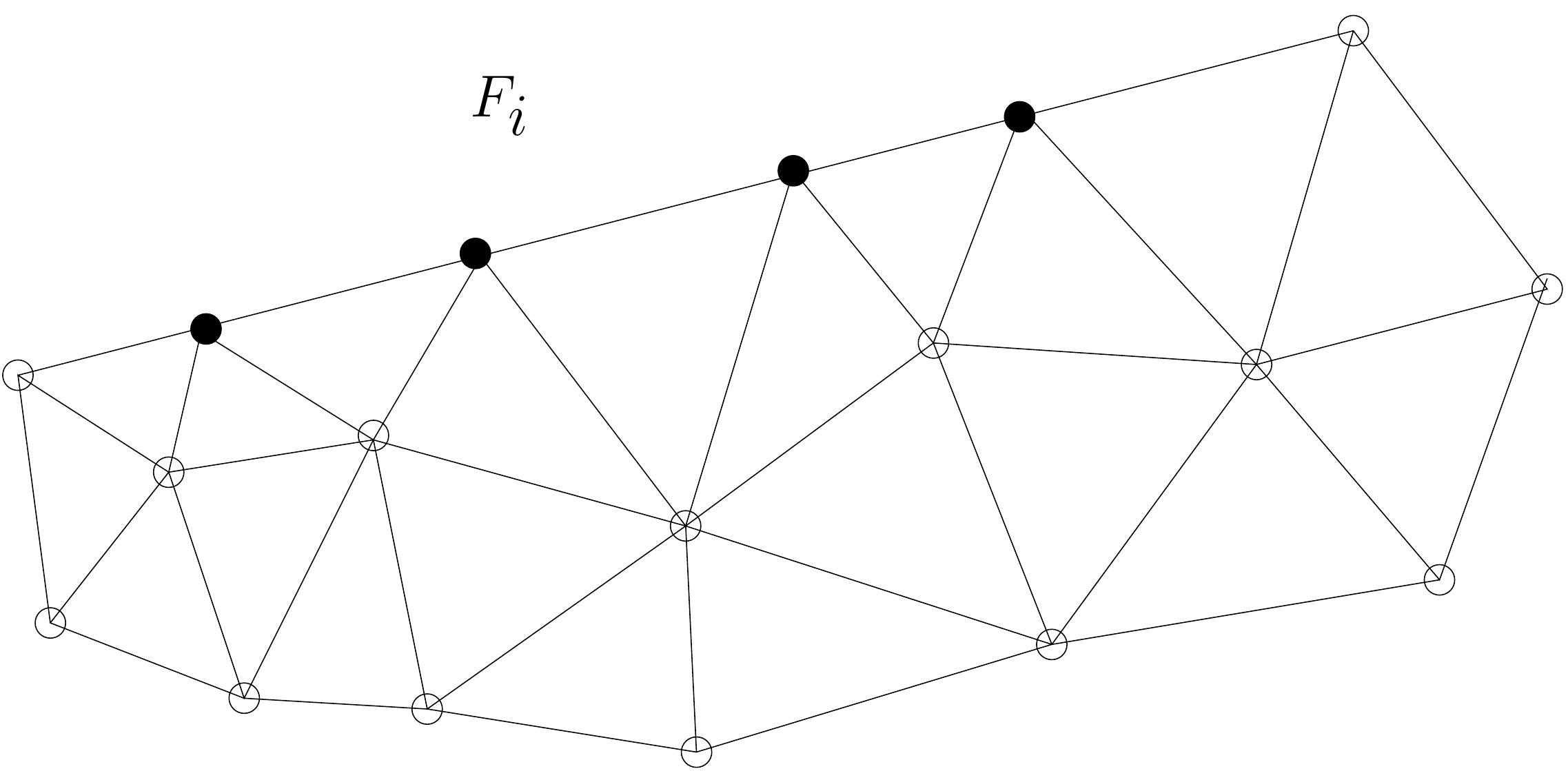}
\caption{Example of a boundary patch $P_i$, with the functions $w_I$
  (left) and $w_F$ (right). The functions take the value $1$ in filled
  nodes and zero in the other nodes.}
\label{patch_CIP}
\end{figure}
\begin{remark}
For high order element the construction of the interpolant $\pi_{CIP}
u$ is much more technical and beyond the scope of the present work. Indeed it is no longer sufficient to prove
orthogonality of $\varphi_i$ against a constant on $P_i$, but it must be
shown to be orthogonal to the continuous finite element space of order
$k-1$, on $P_i$. On the other hand the patches $P_i$ can be chosen
freely, provided $diam(P_i) = O(h)$.
\end{remark}
\section{Numerical examples}
We study two different numerical examples, both have been computed
using the package FreeFem++ \cite{freefem}. First we consider a
simple problem with smooth exact solution, then we
consider a convection-diffusion problem and show the stabilizing
effect of the Nitsche type weak boundary condition for convection
dominated flow. 
\subsection{Problem with smooth solution}
We consider equation \eqref{poisson} in the unit square, with
$f=5\pi^2 \sin(\pi x) \sin(2 \pi y)$ and
$g=0$. The mesh is unstructured with $N=10,20,40,80$ elements per
side. The exact
solution is then given by $u=\sin(\pi x) \sin(2 \pi y)$. We give the
convergence in both the $L^2$-norm and the $H^1$-norm for piece wise affine approximation
in Table \ref{unstr_smooth}. The case of quadratic approximation is
considered in Table \ref{unstr_smooth2}. The order $p$ in $O(h^p)$ is given
in parenthesis next to the error.
\begin{table}
\begin{center}
\begin{tabular}{|l|c|c|c|c|}
\hline
N & Nitsche $H^1$ & strong $H^1$ & Nitsche $L^2$&  strong $L^2$\\
\hline 
10 &7.0E-1 (---)&6.7E-1 (---)&2.4E-2 (---)&2.0E-2 (---)\\
\hline 
20 &3.5E-1 (1.0)& 3.5E-1 (0.94) &5.5E-3 (2.1) & 5.5E-3 (1.9)\\
\hline 
40 &1.7E-1 (1.0) & 1.7E-1 (1.0) & 1.3E-3 (2.1) & 1.3E-3 (2.1)\\
\hline 
80 &8.2E-2 (1.1) & 8.2E-2 (1.1) & 3.3E-4 (2.0) & 3.1E-4 (2.1)\\
\hline
\end{tabular}
\caption{Comparison of errors between the non-symmetric Nitsche method
and standard strongly imposed boundary conditions, using piece wise affine
approximation on unstructured meshes.}
\label{unstr_smooth}
\end{center}
\end{table}
\begin{table}
\begin{center}
\begin{tabular}{|l|c|c|c|c|}
\hline
N & Nitsche $H^1$ & strong $H^1$ & Nitsche $L^2$&  strong $L^2$\\
\hline 
10 & 5.3E-2 (---)&5.1E-2 (---)&1.7E-3 (---)&6.5E-4 (---)\\
\hline 
20 &1.4E-2 (1.9) &1.4E-2 (1.9) &2.2E-4 (2.9) &9.6E-5 (2.8)\\
\hline 
40 &3.5E-3 (2.0)&3.5E-3 (2.0)&2.1E-5 (3.4) &1.1E-5 (3.1)\\
\hline 
80 &8.6E-4 (2.0)&8.6E-4  (2.0)&2.5E-6 (3.1)&1.4E-6 (3.0)\\
\hline
\end{tabular}
\caption{Comparison of errors between the non-symmetric Nitsche method
and standard strongly imposed boundary conditions, using piece wise quadratic
approximation on unstructured meshes.}\label{unstr_smooth2}
\end{center}
\end{table}
%
%
\begin{table}
\begin{center}
\begin{tabular}{|l|c|c|c|c|c|}
\hline
error norm& $\gamma=0$ & $\gamma=10$ & $\gamma=20$ & $\gamma=40$& $\gamma=80$ \\
\hline
$\|u-u_h\|_{L^2}$ & 3.3E-4  & 2.9E-4 & 3.0E-4 & 3.0E-4 & 3.0E-4\\
\hline
$\|u-u_h\|_{H^1}$ & 8.2E-2 &  8.2E-2  &  8.2E-2&   8.2E-2 & 8.2E-2\\
\hline
\end{tabular}
\caption{Study of the dependence of the accuracy on the penalty
  parameter, piece wise affine approximation, unstructured mesh, $N=80$}\label{pen_1}
\end{center}
\end{table}
%
\begin{table}
\begin{center}
\begin{tabular}{|l|c|c|c|c|c|}
\hline
error norm& $\gamma=0$ & $\gamma=10$ & $\gamma=20$ & $\gamma=40$& $\gamma=80$ \\
\hline
$\|u-u_h\|_{L^2}$ & 2.1E-5  &  1.3E-5& 1.2E-5 & 1.2E-5 & 1.2E-5\\
\hline
$\|u-u_h\|_{H^1}$ & 3.5E-3 & 3.5E-3 & 3.5E-3 &  3.5E-3 & 3.5E-3\\
\hline
\end{tabular}
\caption{Study of the dependence of the accuracy on the penalty
  parameter, piece wise quadratic approximation,  unstructured mesh, $N=40$}\label{pen_2}
\end{center}
\end{table}
We have not managed to construct an example exhibiting the suboptimal
convergence order of the Nitsche method. Some cases with
non-homogeneous boundary conditions, not reported here, were computed
both with affine and quadratic elements. They all had optimal
convergence on the finer meshes.
The theoretical results do not extend
to the symmetric version of Nitsche's method and stability is unlikely
to hold on general meshes. Applying the symmetric method to
the proposed numerical example yields a solution with clear boundary
oscillations on the coarse meshes see Figure \ref{sym_vs_nonsym}. On
finer meshes these oscillations vanish and the performance is
similar to that of the non-symmetric method.
\begin{figure}
\begin{center}
\hspace{-1.cm}
\includegraphics[width=7cm]{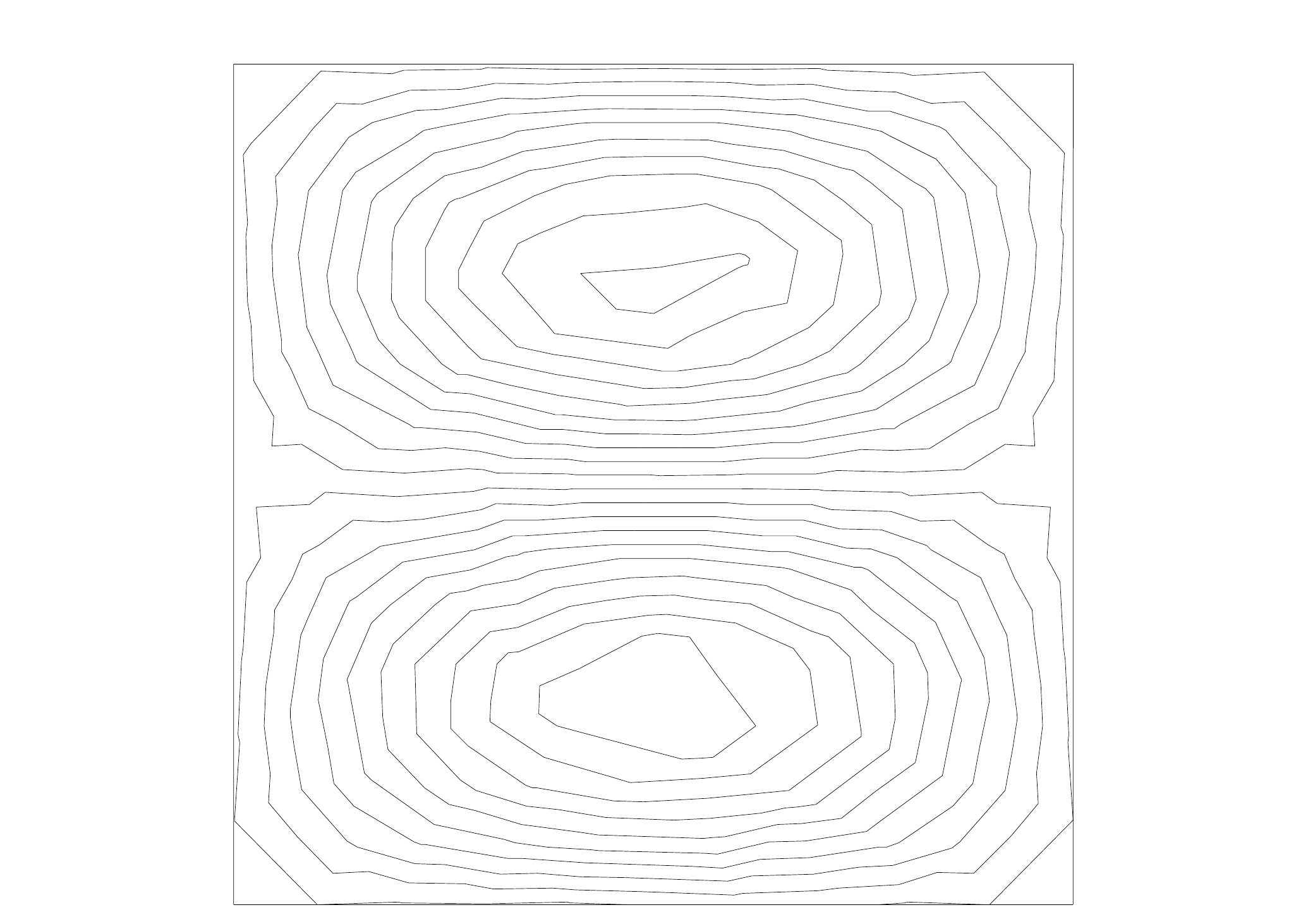}\hspace{-1.7cm}
\includegraphics[width=7cm]{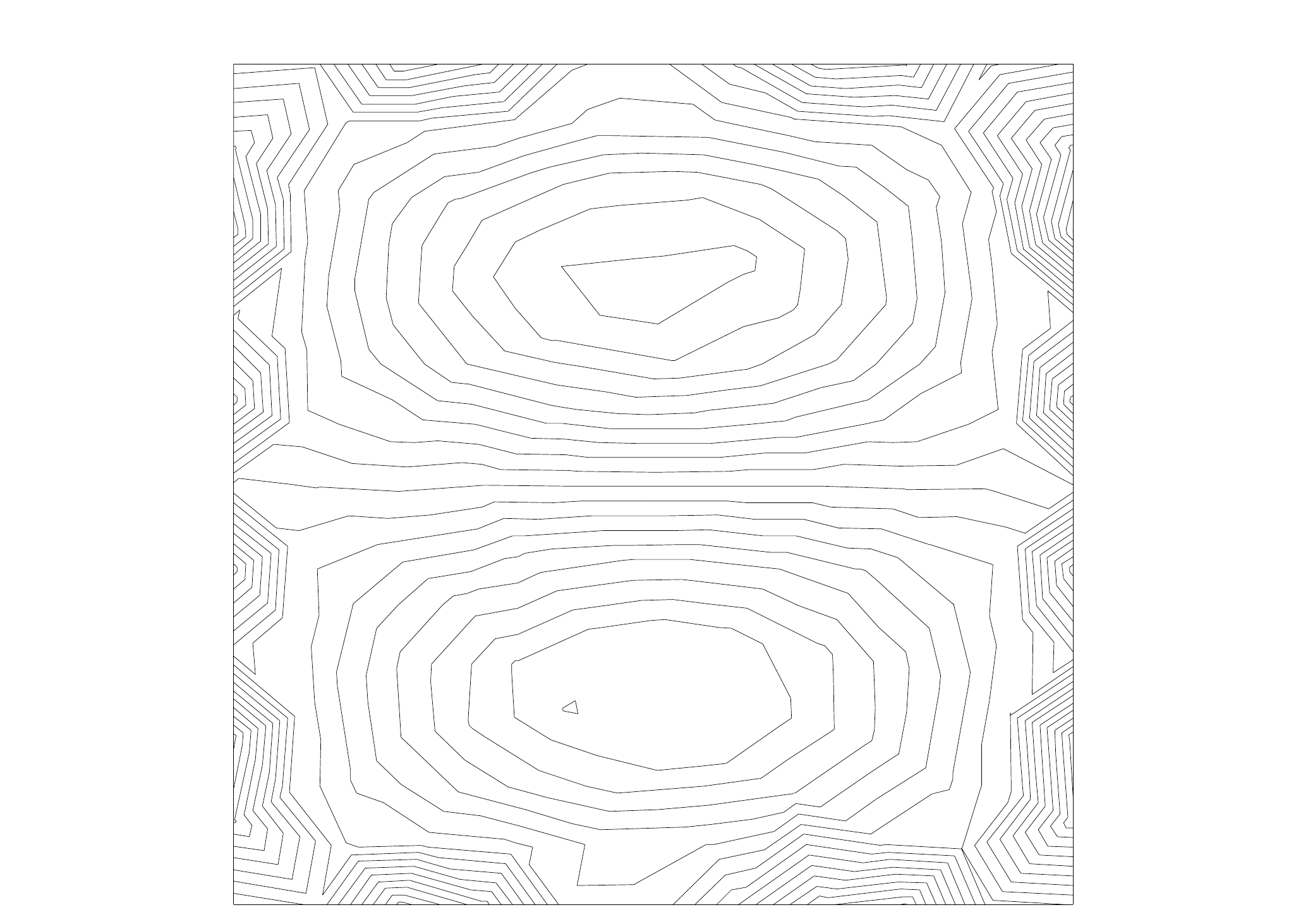}\hspace{-1.7cm}
\caption{Comparison of the contourplots of the unstabilized
  non-symmetric method (left) and symmetric (right) method, piece wise
  affine approximation, N=10.}\label{sym_vs_nonsym}
\end{center}
\end{figure}
Note that although the convergence of the Nitsche method is optimal in
this case, the error
constant of the non-symmetric method in the $L^2$-norm is a factor two larger than that of
the strongly imposed boundary conditions for piece wise quadratic
approximation. The same computations were made on structured meshes
(not reported here) and this effect was slightly larger in this case,
with a factor two in the affine case and four in the quadratic case.
The errors in the $H^1$-norm on the other hand are
of comparable size for the two methods.

 This motivates a study of how the error depends on the penalty
 parameter $\gamma$ in \eqref{penalty}. We therefore run a series of
 computations
with $\gamma=0,10,20,40,80$. In Table
\ref{pen_1}  we report the results for piece wise affine approximation
and in Table \ref{pen_2}  the results for piece wise quadratic 
approximation. 
We note that there is a visible, but negligible,
effect on the error measured in the $L^2$-norm, but no 
effect on the error in the $H^1$-norm.
\subsection{Problem with outflow layer}
For this case we only compare the solutions qualitatively. We
consider the problem with a convection term \eqref{conv_diff}.
To create an outflow layer we have chosen $f := 1$, $\beta := (0.5,1)$, $\sigma := 0$ in $\Omega$.
We discretized $\Omega$ with a structured mesh having $80$ piece wise affine elements on each side. The
contourplots for $\varepsilon = 0.1, 0.001, 0.00001$ are reported in
Figure \ref{nit_lay} for Nitsche's method and in Figure \ref{gal_lay}
for the strongly imposed boundary conditions. Note that no stabilization has been
added in either case. This computation illustrates the strong stabilizing effect
of the weakly imposed boundary condition. A theoretical explanation of
this phenomenon was given in \cite{Sch08}. Finally we consider the
effect of adding stabilization to the computation. In this case we
take $N=80$ with piece wise quadratic approximation. We report the
results
of a computation without stabilization, with the SD-method
($\gamma_{SD} = 0.2$) and with
the CIP-method ($\gamma_{CIP}=0.005$) in Figure \ref{nit_stab}. Note that the stabilized
methods
clean up the remaining spurious oscillations in both cases.
\begin{figure}
\begin{center}
\hspace{-1.cm}
\includegraphics[width=5.3cm]{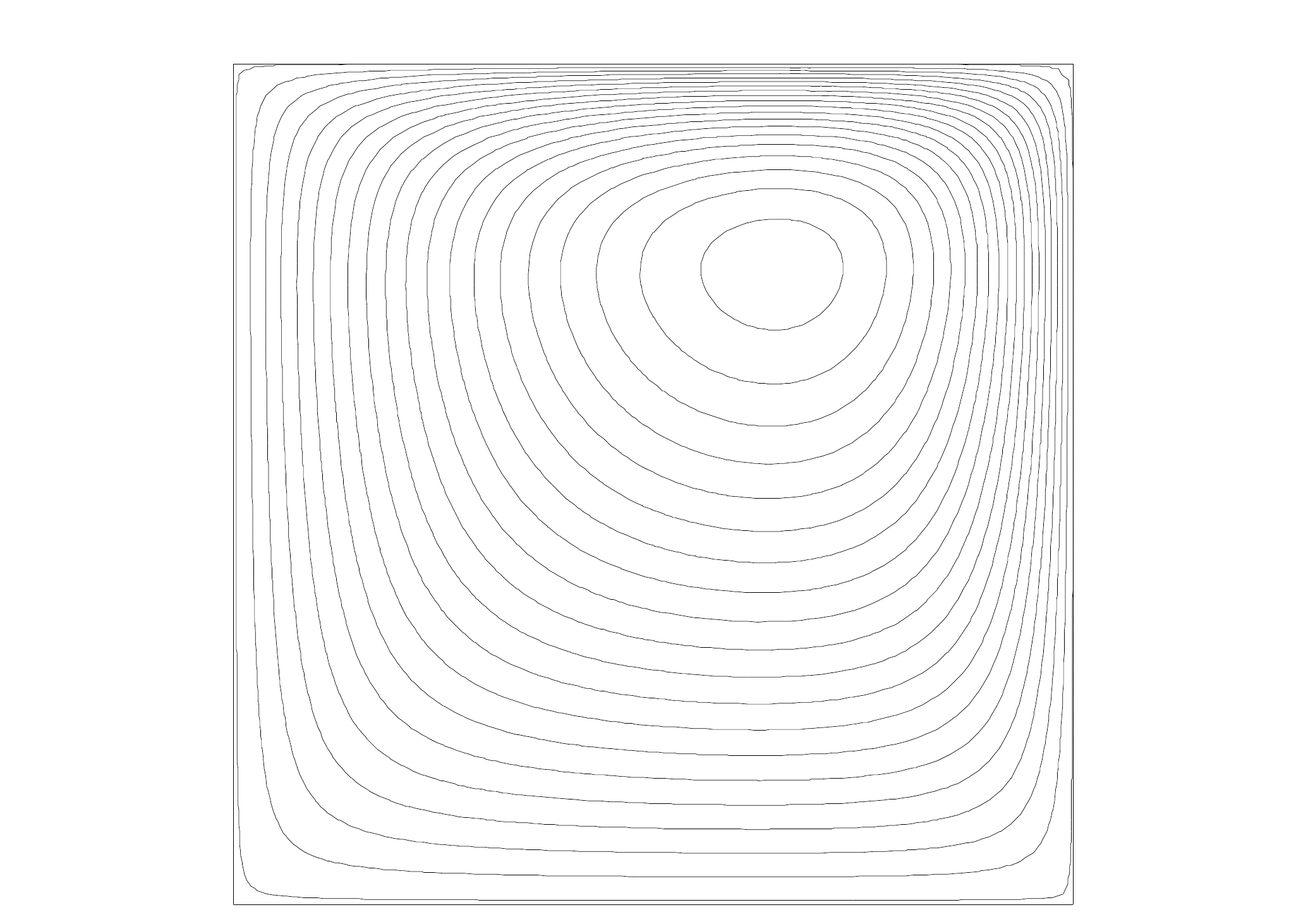}\hspace{-1.7cm}
\includegraphics[width=5.3cm]{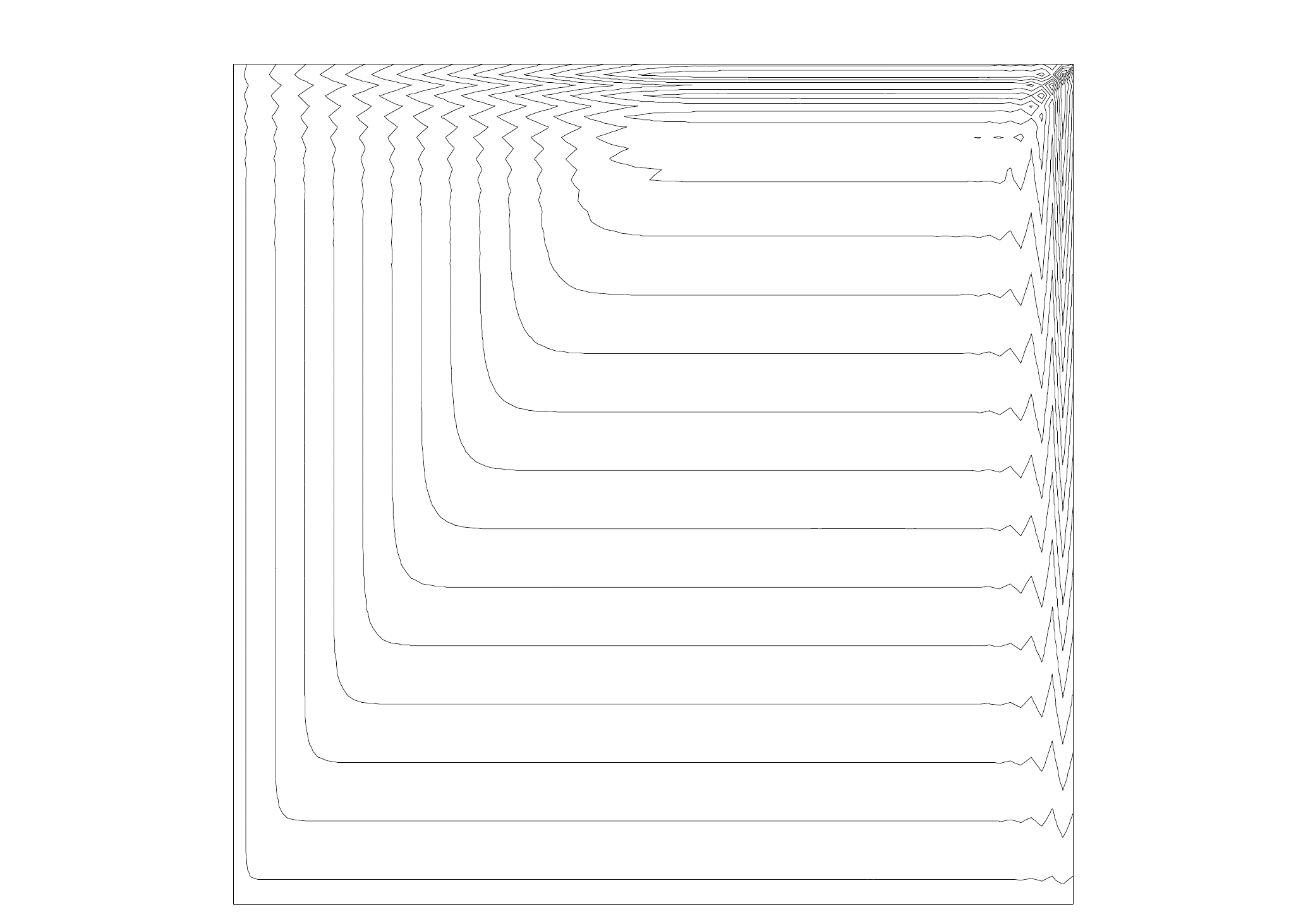}\hspace{-1.7cm}
\includegraphics[width=5.3cm]{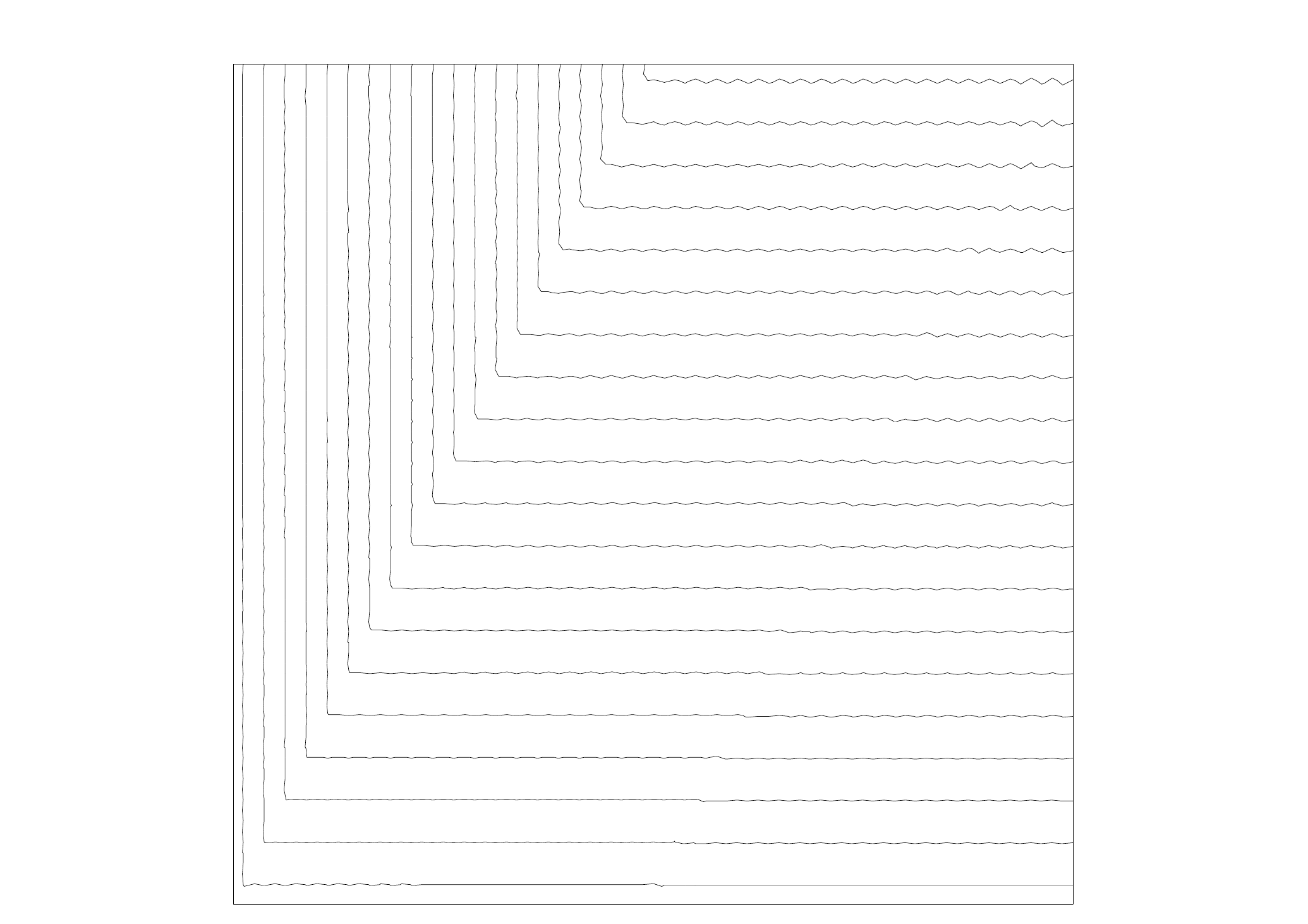}
\caption{Convection-diffusion equation discretized using the
  non-symmetric Nitsche boundary condition, $N=80$, piece wise affine
  approximation, from left to right:
  $\varepsilon = 0.1$, $\varepsilon = 0.001$, $\varepsilon =
  0.00001$. }\label{nit_lay}
\end{center}
\end{figure}
\begin{figure}
\begin{center}
\hspace{-1.cm}
\includegraphics[width=5.3cm]{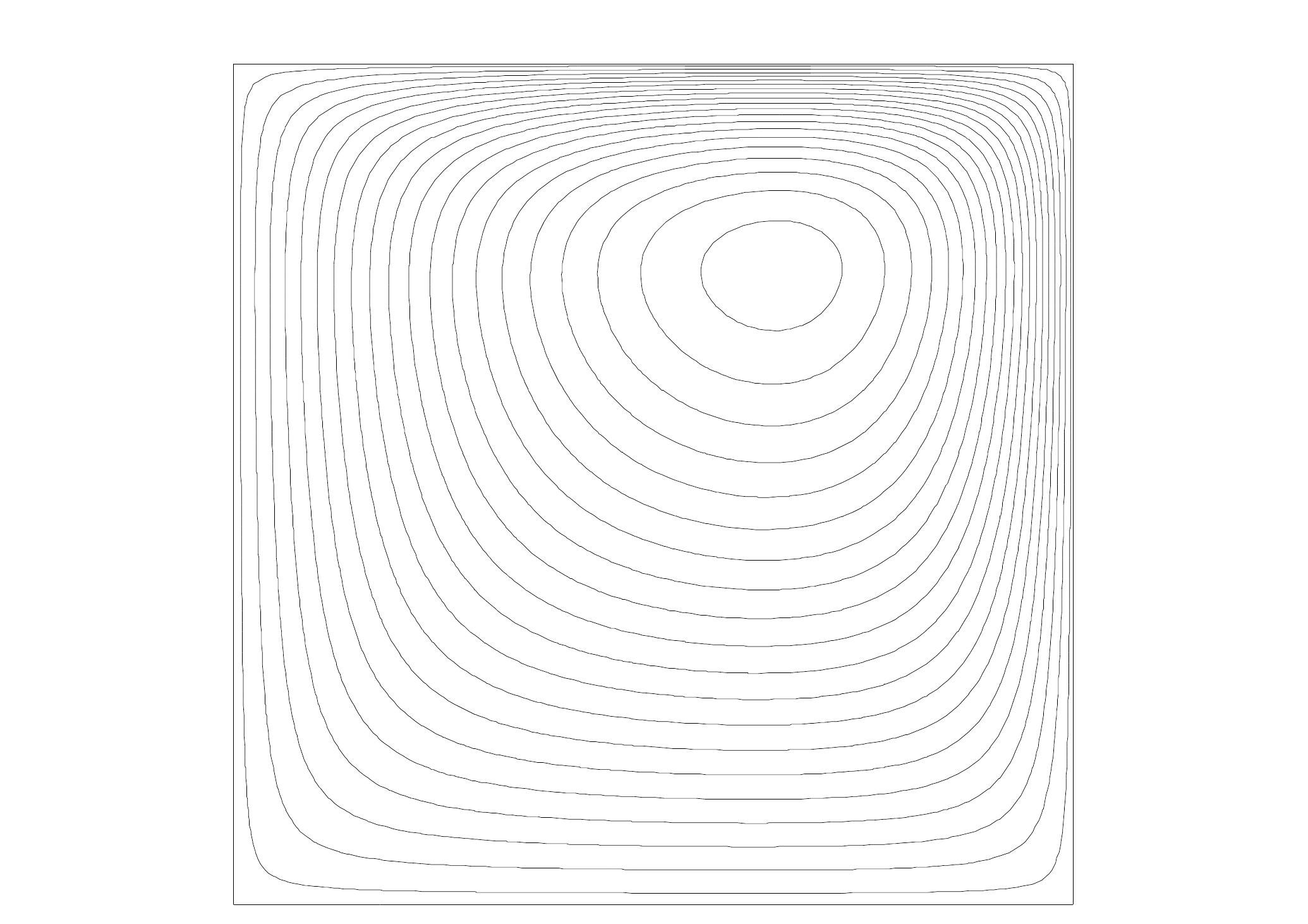}\hspace{-1.7cm}
\includegraphics[width=5.3cm]{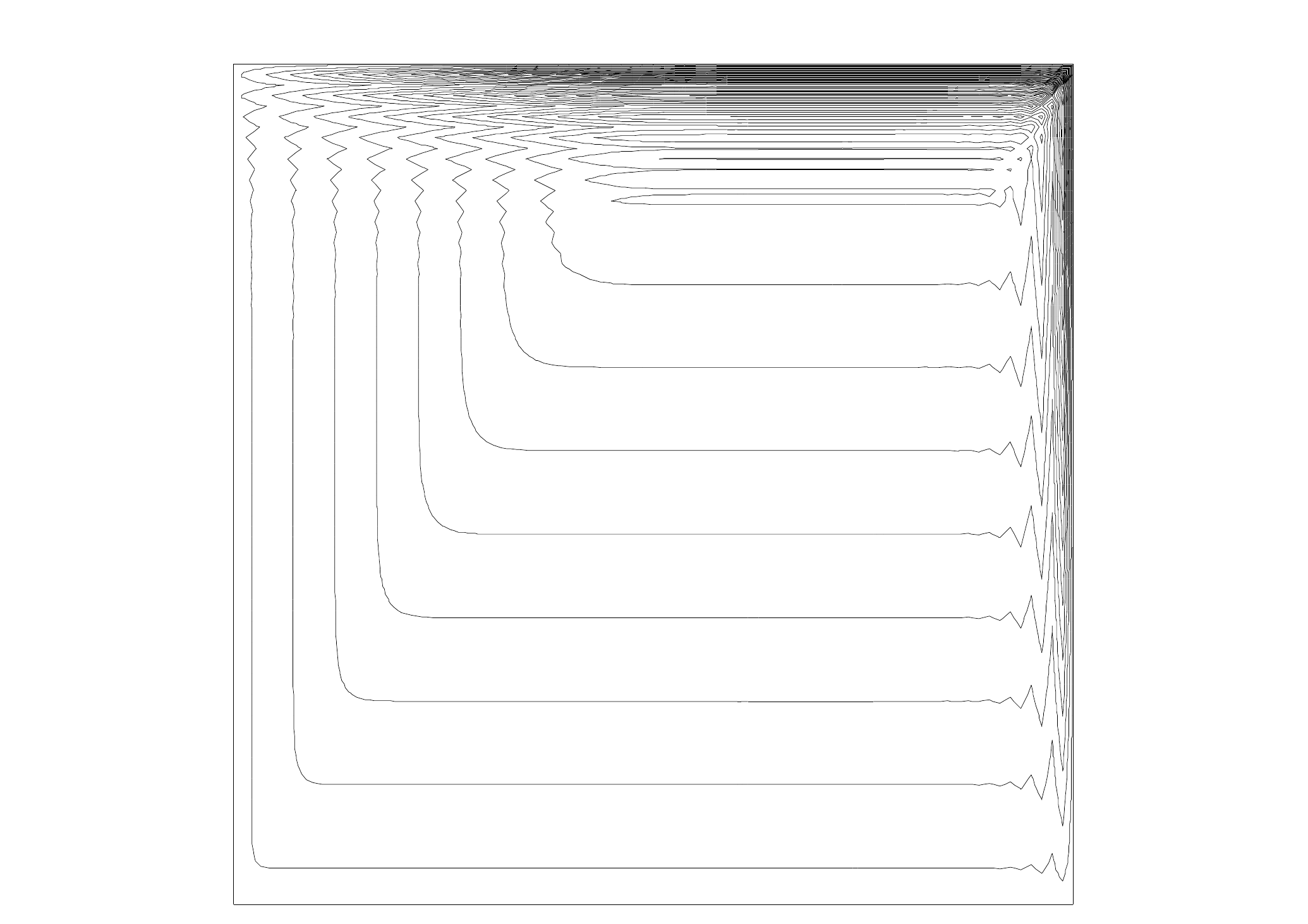}\hspace{-1.7cm}
\includegraphics[width=5.3cm]{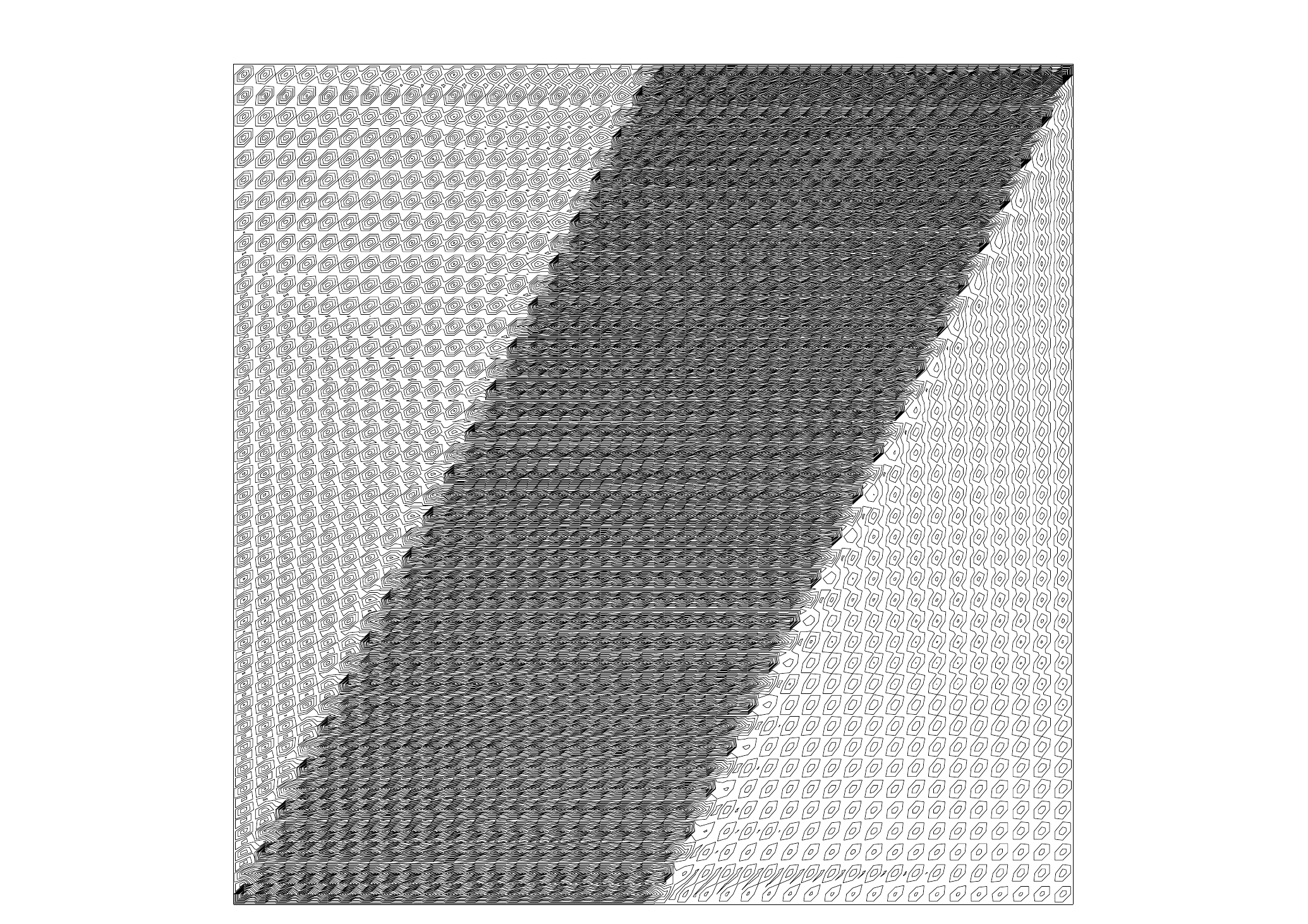}
\caption{Convection-diffusion equation discretized using strongly
  imposed boundary condition, $N=80$, piece wise affine approximation, from left to right:
  $\varepsilon = 0.1$, $\varepsilon = 0.001$, $\varepsilon =
  0.00001$.}\label{gal_lay}
\end{center}
\end{figure}
\begin{figure}
\begin{center}
\hspace{-1.cm}
\includegraphics[width=5.3cm]{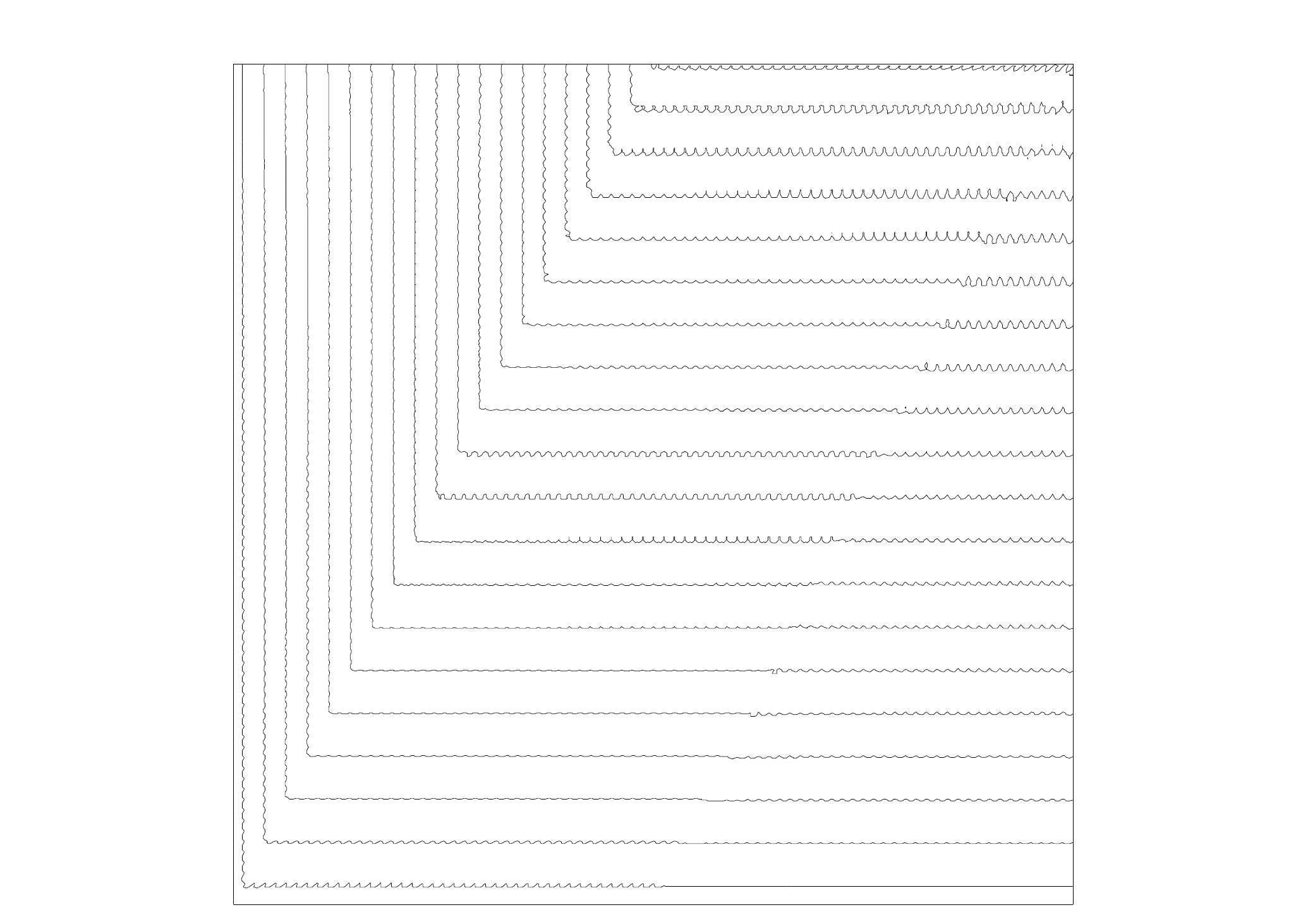}\hspace{-1.7cm}
\includegraphics[width=5.3cm]{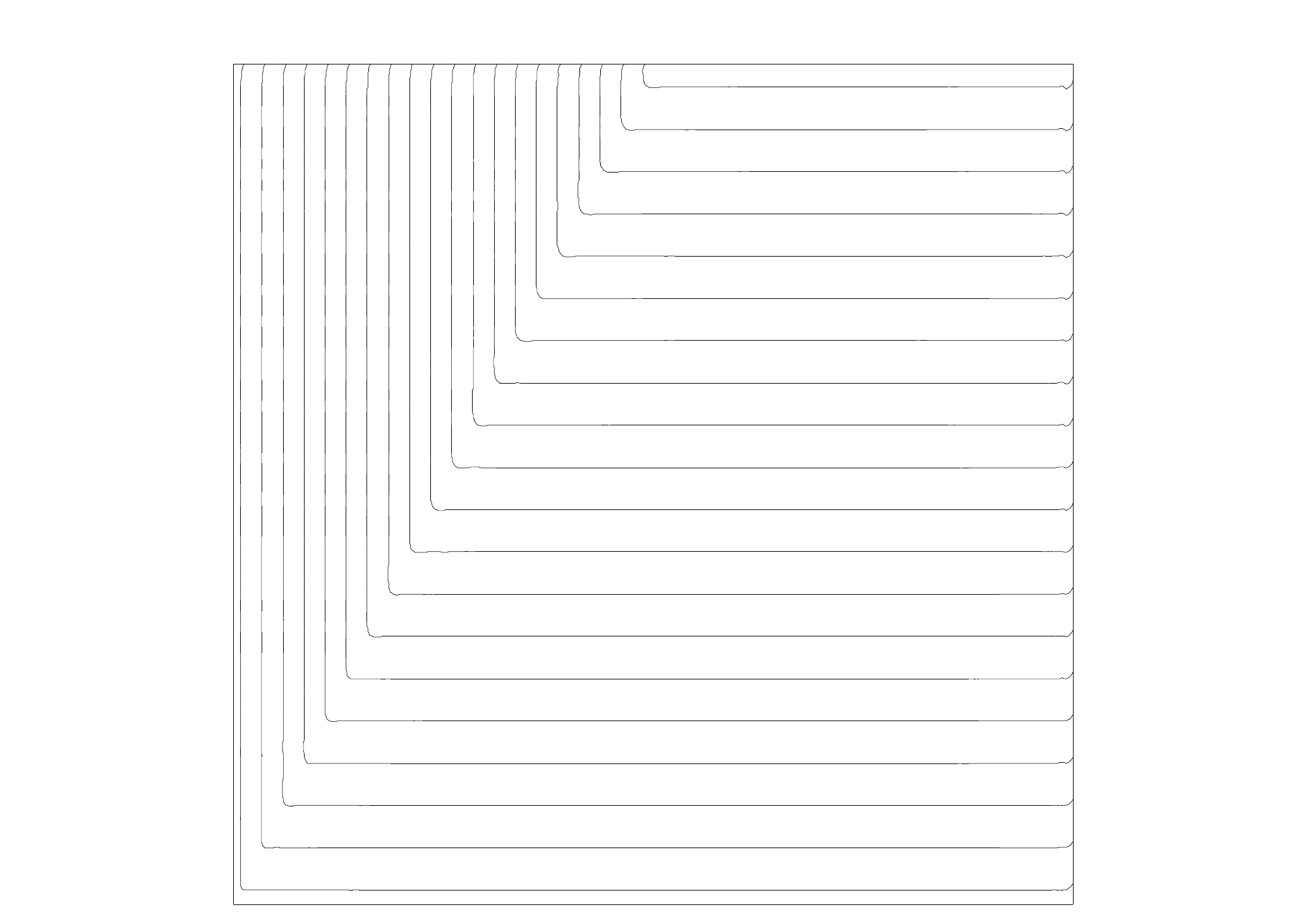}\hspace{-1.7cm}
\includegraphics[width=5.3cm]{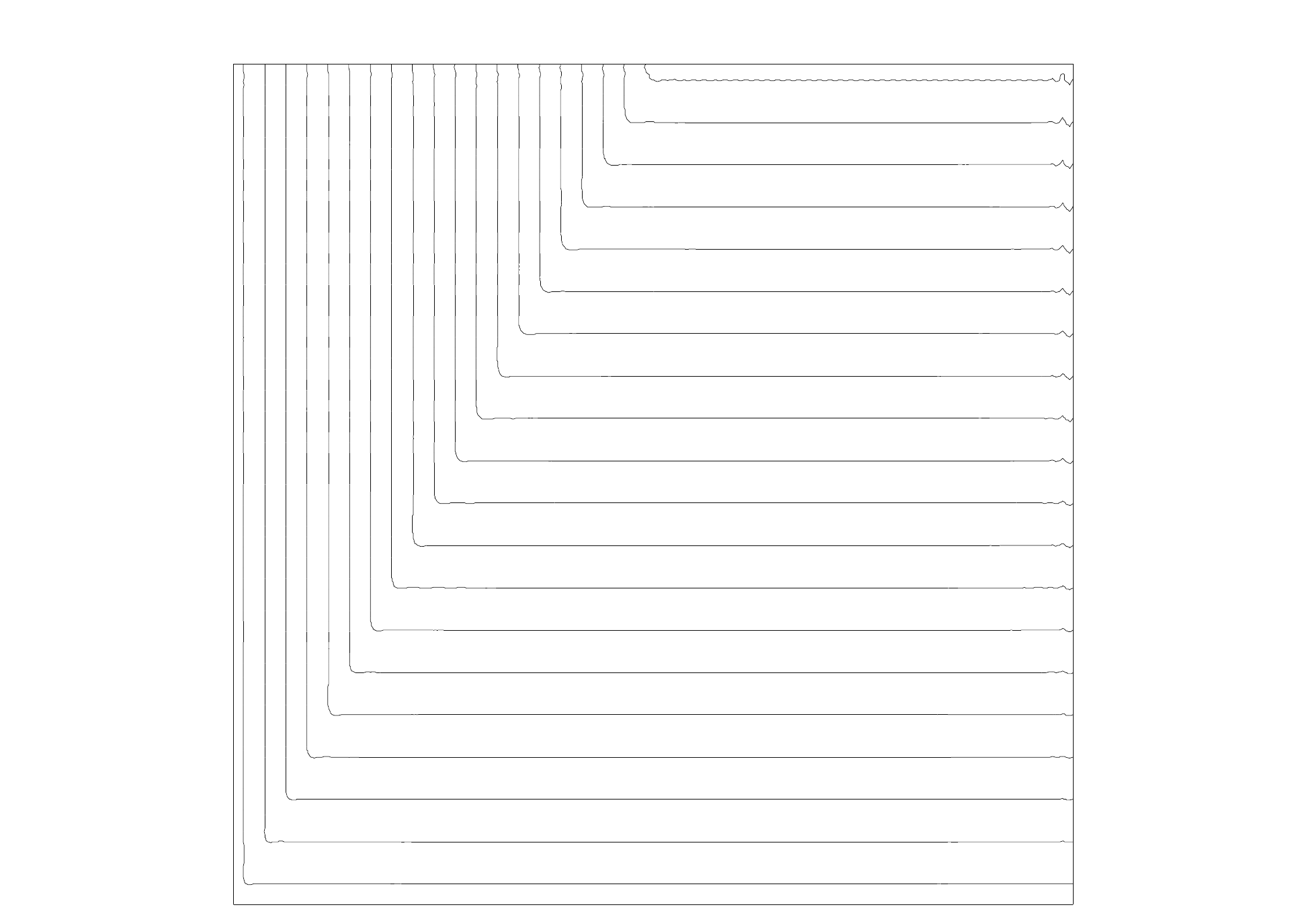}
\caption{Convection-diffusion equation discretized using the
  non-symmetric Nitsche boundary condition, $N=80$, $\varepsilon =
  0.00001$, piece wise
  quadratic approximation, from left to right:
  no stabilization, SD-stabilization ($\gamma_{SD} = 0.5$),
  CIP-stabilization ($\gamma_{CIP}=0.005$). }\label{nit_stab}
\end{center}
\end{figure}
\\[5mm]
{\bf Acknowledgements }
This note would not have been written without Professor Tom Hughes who
told me that the non-symmetric Nitsche's method appeared to be stable without
penalty in large-eddy simulations and pointed me to the reference \cite{HEML00b}. I would also like to thank
Professor Rolf Stenberg for interesting
discussions
on the subject of Nitsche's method.
\bibliographystyle{plain}   
\bibliography{Biblio}
\end{document}